\documentclass[11pt]{article}

\usepackage[english]{babel}
\usepackage[utf8]{inputenc}
\usepackage[T1]{fontenc}
\usepackage{geometry}
\usepackage{lmodern}
\usepackage{microtype}
\usepackage{amsmath,amssymb,amsthm,mathtools}
\usepackage{bm}
\usepackage{bbm}
\usepackage{enumitem}
\usepackage[colorlinks=true,linkcolor=blue,citecolor=blue,urlcolor=blue]{hyperref}
\usepackage{algorithm}
\usepackage{algorithmicx}  
\usepackage{algpseudocode}
\usepackage{float}
\usepackage{natbib}
\usepackage{csquotes}
\usepackage[dvipsnames]{xcolor}

\numberwithin{equation}{section}

\theoremstyle{plain}
\newtheorem{theorem}{Theorem}[section]
\newtheorem{proposition}[theorem]{Proposition}
\newtheorem{lemma}[theorem]{Lemma}

\newtheorem{corollary}[theorem]{Corollary}

\theoremstyle{definition}
\newtheorem{definition}[theorem]{Definition}
\newtheorem{assumption}[theorem]{Assumption}

\theoremstyle{remark}
\newtheorem{remark}[theorem]{Remark}

\def\cA{{\mathcal A}}

\def\cC{{\mathcal C}}

\def\cE{{\mathcal E}}

\def\cG{{\mathcal G}}
\def\cH{{\mathcal H}}

\def\cP{{\mathcal P}}
\def\cR{{\mathcal R}}

\def\cU{{\mathcal U}}

\def\cT{{\mathcal T}}
\def\cW{{\mathcal W}}
\def\cX{{\mathcal X}}
\def\cY{{\mathcal Y}}

\def\E{\mathbb{E}}

\def\N{\mathbb{N}}

\def\R{\mathbb{R}}

\def\fE{\mathfrak{E}}

\def\eps{\epsilon}

\def\d{\mathrm{d}}

\newcommand{\mc}{\mathcal}

\usepackage{fancyhdr}
\title{Implicit Regularization of Large Neural Networks via\\ Mean-Field Formulation}
\author
{Beatrice Acciaio\footnote{ Department of Mathematics, ETH Z\"urich, Z\"urich, Switzerland, email: 
{\tt beatrice.acciaio@math.ethz.ch}.}
\and Jakob Heiss \footnote{Statistics Department, UC Berkeley, Berkeley, USA, email: 
{\tt jakob.heiss@berkeley.edu}.}
\and Gudmund Pammer\footnote{Institute for Statistics, TU Graz, Styria, Austria, email: 
{\tt gudmund.pammer@tugraz.at}.}
\and Qinxin Yan\footnote{Program in Applied and Computational
Mathematics, Princeton University, Princeton, USA, email: 
{\tt qy3953@princeton.edu}. }}

\date{\today}

\begin{document}
\maketitle
\begin{abstract} 
\noindent 
We propose a mathematical framework to explain implicit regularization from early stopping during the training of overparametrized neural networks. In the mean-field limit, the parameter distribution evolves according to a gradient flow on the space of probability measures. We show that these dynamics admit an equivalent McKean–Vlasov stochastic control formulation through the corresponding Hamilton–Jacobi–Bellman (HJB) equation. The control viewpoint yields a Dynamic Programming Principle (DPP), which we use to define a new metric on probability measures. This metric can be viewed as a mean-field generalization of the control representation of the Wasserstein-2 distance, and it naturally appears as a regularization term selected by early stopping. We further obtain non-asymptotic bounds describing how the induced regularization depends on the stopping time.
\end{abstract}
\vspace{3pt}

\vspace{3pt}

\section{Introduction}\label{sec:1}
An important empirical feature of modern deep learning is that highly overparametrized models can fit the
training data extremely well, often achieving (near) interpolation, and yet still generalize
surprisingly well on unseen data 
\citep{ZhangBengioHardtRechtVinyals2017rethinking,BelkinHsuMaMandal2019doubledescent}.
From a traditional statistical viewpoint, this behavior is counterintuitive: increasing the capacity of a
hypothesis class is usually expected to worsen generalization unless one imposes explicit regularization.
Understanding why and when large neural networks avoid classical overfitting has therefore become a
central question in machine learning theory.

A broad set of explanations has emerged. Some lines of work emphasize the role of the data distribution
and the emergence of useful internal representations \citep{BengioCourvilleVincent2013representation}, others connect
generalization to properties of the optimization dynamics, such as algorithmic stability (and its link
to training time) \citep{HardtRechtSinger2016} or margin/Lipschitz-type complexity measures
\citep{NIPS2017_b22b257a}.
A particularly influential viewpoint is \emph{implicit regularization} (or implicit bias): even when the
training objective contains no explicit regularization like Lasso or Ridge, the optimization algorithm can preferentially select
certain \enquote{simpler} solutions among the many global minimizers, thereby controlling effective complexity
\citep{Neyshabur2014InSO,SoudryImplicitBiasSeparableData2017arXiv171010345S,pmlr-v80-gunasekar18a}.
In this perspective, hyperparameters and training protocols, including noise injection, step sizes, and
especially the stopping time, act as regularization mechanisms. Early stopping is a canonical example:
terminating the optimization trajectory at a finite time can select solutions with improved stability
or generalization properties even in highly nonconvex regimes.

In this paper we propose a mathematical framework for implicit regularization and early stopping based
on a mean-field description of wide neural networks. In the mean-field (infinite-width) regime, a
two-layer network can be viewed as a function on a probability measure space \citep{MeiMFNN, chizatbach,SirignanoNN,RotskoffNN}. Training by gradient-based methods then induces an evolution of probability measures
$(\mu_t)_{t\ge0}$ on the parameter space. This evolution can be interpreted as a gradient flow on the probability measure space  endowed with the Wasserstein-2 metric. This links optimization dynamics to the geometry of optimal
transport, variational principles, and to partial differential equations.

Our main contribution is to connect this gradient-flow viewpoint with a McKean--Vlasov optimal control
formulation and its Hamilton--Jacobi--Bellman (HJB) equation  on probability measure space. This bridge allows us to
interpret training trajectories---and in particular early-stopped trajectories---as solutions selected
by a variational principle that trades off loss decrease against a kinetic (control) cost. We develop
both a deterministic regime and a diffusive regime,
derive endpoint variational representations consistent with dynamic programming, and use these tools
to formalize early stopping as an implicit regularizer in the mean-field setting. 

\paragraph{Our contributions.}
At a high level, we establish a two-way bridge between Wasserstein gradient flows and McKean--Vlasov
control for training dynamics on probability measure space, and we use it to quantify and interpret early-stopping
bias. Concretely:
\begin{itemize}
    \item We formulate deterministic and diffusive mean-field training dynamics as minimizing curves of McKean--Vlasov control
    problems and characterize their value functions via stationary HJB equations on probability measure space.
    \item We derive endpoint (finite-horizon) variational formulations for the training dynamics that
    are compatible with dynamic programming. These yield a principled description of early stopping as
    a trajectory-selection mechanism.
    \item We develop a geometric structure through continuity equations, minimal velocities, and a
    Finsler-type action, clarifying how the selected trajectory balances loss dissipation and kinetic
    energy, and how this balance changes between deterministic and diffusive regimes.
    \item We apply the framework to the mean-field model of one-hidden-layer neural networks, establish
    moment and stability bounds 
    and illustrate the resulting implicit-regularization behavior numerically.
    \item We further use our framework to provide an alternative solution for this open question:
    \begin{quote}
        \textit{Does implicit regularization vanish under infinite training time horizon?}
    \end{quote}
\end{itemize}

\paragraph{Paper organization.}
Section~\ref{sec:2} recalls some preliminaries on gradient flows on metric spaces, and introduces the McKean--Vlasov control formulation.
Section~\ref{sec:3} develops endpoint representations and the early-stopping variational
principles. Section~\ref{sec:4}  specializes the general results to one-hidden-layer
networks and derives the main bounds and interpretations. Numerical experiments are presented in
Section~\ref{sec:numerics}.

\subsection{Related Literature}\label{sec:literature}

\citet{bishop1995regularizationEarlyStopping,friedman2003gradient,pmlr-v89-ali19a,heiss2019implicit1,Stark_Steinberger_2025,Wu_Bartlett_Lee_Kakade_Yu_2025} study implicit regularization for linear regression (including random feature models), by connecting the (early-stopped) training dynamics to the $L_2$-regularization of the parameters, but their theory does not apply to deep neural networks with trainable hidden layers.
\citet{GidelImplicitDiscreteRegularizationDeepLinearNN2019arXiv190413262G} analyze implicit regularization for deep linear neural networks, where all the activation functions are linear.

In contrast, we study the gradient dynamics and the implicit regularization of early-stopping for a two-layer\footnote{In our notation, a two-layer neural network has one hidden layer (with non-linear activation function). For non-polynomial activation functions this architecture can approximate any continuous function on any compact set with sufficiently many neurons \citep{CybenkoUniversalApprox1989,HornikUniversalApprox1991251}.} neural network with non-linear activation function in the mean-field limit. The mean-field limit corresponds to the two-layer special case of  \citep{muPpmlr-v139-yang21c} which is in the \enquote{feature learning} regime in contrast to the \enquote{lazy learning} regimes such as NTK \citep{jacot2020neuraltangentkernelconvergence}.
In particular, our analysis allows us to bound the $L^2$-norm of the parameters obtained from early-stopped gradient flow. Such a bound is particularly useful for ReLU neural networks, since multiple works in the literature translate the $L^2$-norm of the parameters into properties of the learned function \citep{savarese2019infinite,ongie2019function,heiss2021reducing,HeissPart3Arxiv,parhi2022kinds,jacot2022feature,Boursier_Flammarion_2023}, which can provide more fine-grained understanding of the inductive bias towards feature learning \citep{HeissPart3Arxiv,jacot2022feature,HeissInductiveBias2024,MultiOutputSimilarToHeiss2021JMLR:v25:23-0677,Parkinson_Ongie_Willett_2025}. Therefore, we think that there is a potential for future research to better understand inductive bias towards feature learning for early stopping. However, within this work, we do not further discuss the connections to feature learning.
While we highlight this potential connection to feature learning for future work, the primary focus of this work remains on the variational characterization of the training dynamics itself.

\citet{williams2019gradient} analyze the gradient dynamics of univariate ReLU networks, demonstrating how optimization adaptively aligns activation thresholds (knots) with the data structure. From a margin-maximization perspective, \citet{PoggioGernalizationDeepNN2018arXiv180611379P} argue that gradient descent implicitly regularizes deep networks by maximizing the normalized margin, akin to support vector machines.
\cite{paik2025basicinequalitiesfirstorderoptimization} provide more general bounds for the implicit regularization of general optimization problems without exploiting the structure of neural networks.

Our analysis relies on the intersection of gradient flows in metric spaces, McKean--Vlasov control, and viscosity solutions for Hamilton-Jacobi-Bellman (HJB) equations on the Wasserstein space. The interpretation of the training dynamics of infinite-width neural networks as a gradient flow on the space of probability measures is popularized by \citep{chizatbach,MeiMFNN,SirignanoNN,RotskoffNN}. Rigorous foundations for such flows are established in the theory of curves of maximal slope in metric spaces, comprehensively developed by \cite{ambrosioGradientFlowsMetric2008}.

To characterize the implicit regularization effect of early stopping, we reformulate the gradient flow via a McKean--Vlasov control problem, where the state dynamics (the distribution of parameters) depend on the distribution itself. A similar idea can be found in~\cite{RossiSavareMielke2008}, where they prove the existence of curves of maximal slope on general metric space using control formulations and corresponding Hamilton-Jacobi equations on general metric spaces. On probability measure space, the control problem is referred to as McKean--Vlasov control (equivalently Mean-field control), and the general theory for McKean--Vlasov control problems is detailed in \cite{Carmona2018Probabilistic}. A central tool in our analysis is the Dynamic Programming Principle (DPP), which allows us to decompose the infinite-horizon problem into finite-time segments. While classical DPP results are standard, the extension to McKean--Vlasov dynamics, where the conditioning on the law, induces non-trivial technicalities. We rely specifically on the results of \cite{DylanDPP, PossamaControl}, who prove the DPP for open-loop controls and justify the restriction to feedback controls in the weak formulation.

The value function of the McKean--Vlasov control problem is characterized as the solution to a Hamilton-Jacobi-Bellman (HJB) equation on the Wasserstein space. Due to the potential lack of smoothness of the value function and the non-compactness of the space, classical smooth solutions rarely exist. We therefore adopt the notion of viscosity solutions on metric spaces. In the noisy training regime, we refer to the work of \citep{SonerYan,Cosso2022MasterBE}, who establish existence and uniqueness comparison principles for McKean--Vlasov control problems with diffusion. In the deterministic case, the well-posedness and uniqueness of viscosity solutions in this setting are guaranteed by the comparison principles established by \citep{MetricvissolGanbo,AmbFen14}. These works ensure that the value function of the McKean--Vlasov control problem is the unique viscosity solution of the corresponding HJB equation, even when the gradient flows are non-unique due to the lack of displacement convexity or sufficient regularity assumptions.

\subsection{Notations}
Let $E\subseteq \R^d$  denote the underlying space, and $\mathcal P(E)$ the set of probability measures on it. Let $\mathcal P_2(E)$ be the set of probability measures with finite second moment, endowed with the Wasserstein-2 metric $\cW_2$. Recall that $(\cP_2(E), \cW_2)$ is a complete metric space. 

For two metric spaces $\cX, \cY$, let $\cC(\cX;\cY)$ denote the set of continuous functions from $\cX$ to $\cY$ with respect to their own topology. We denote the set of continuous functions on $\cP_2(E)$ by $\cC(\cP_2(E);\R)$, and write $\cC(\cP_2(E))$ when there is no ambiguity. We will use bold symbols to denote continuous curves in $\cP_2(E)$, e.g. $\boldsymbol{\mu}\in \cC(\R_+;\cP_2(E))$. For $i\in \N_+$, let $\cC^i(E;\R)$ denote the set of functions with $i$-times continuous derivative.  We will use the linear derivative on $\mathcal P(E)$ defined as follows.
\begin{definition}
    A function $\phi\in\cC(\cP(E);\R)$  is continuously differentiable if there exists a function $\partial_\mu \phi\in \cC(\cP_2(E);\cC(E))$  satisfying
        $$   \phi(\nu)=\phi(\mu)+\int_0^1\int_E\partial_\mu\phi(\mu+\tau(\nu-\mu))(x)(\nu-\mu)(\d x)\d \tau, \quad \forall \mu,\nu\in \cP(E),
        $$
        where $\partial_\mu\phi(\mu)$ is called  the \emph{linear derivative} of $\phi$ in the $\mu-$variable evaluated at $\mu$. For $i=1,2$, set $\cC^i(\cP(E)):=\{\phi\in \cC(\cP(E)):\partial_\mu \phi\in \cC(\cP(E);\cC^i(E;\R))\}$. If it exists, $\nabla_W \phi:=\nabla_x\partial_\mu \phi:\cP(E)\rightarrow\cC(E;\R^d)$ is called the \emph{Lions derivative}.
\end{definition}

\section{Problem Setup}\label{sec:2}
\subsection{Preliminaries}
We first recall some definitions and results on gradient flows in Wasserstein space from~\citep{ambrosioGradientFlowsMetric2008}. Consider the complete metric space $(\mathcal P_2(E), \cW_2)$. Let $I\subset \R_+$ be an interval, which can be either finite or infinite.
 For $p\in [1,\infty]$, a curve $\boldsymbol{\mu}:=(\mu_t)_{t\in I}$ is said to belong to $ AC^p(I; \mathcal P_2(E))$ if there exists $\boldsymbol{m}= (m_t)_{t\in I}\in L^p(I; \R)$ such that 
\begin{equation}
\label{eq: absolute continuity}
    \cW_2(\mu_s,\mu_t)\leq \int_s^t m_u\, \mathrm{d} u,\qquad \text{for all }\, s,\,t\in I,\, s\le t,
\end{equation}
and $AC_{loc}^p(I;\cP_2(E))$ denotes the set of curves that belong to $AC^p(I_0;\cP_2(E))$ for every compact interval $I_0\subseteq I$. When there is no ambiguity on the target space, we also write $AC^p(I)$. For $p = 1$, $AC^1(I;\cP_2(E))$ is the space of absolutely continuous curves, and we 
denote the corresponding space simply by $AC(I;\cP_2(E))$.

The following results can be found in \citep{ambrosioGradientFlowsMetric2008}.

\begin{proposition}[Metric derivative]
    For all $\boldsymbol{\mu}\in AC^2(I)$, the limit 

    \begin{equation*}
        |\mu'|(t):=\lim_{s\rightarrow t} \frac{\cW_2(\mu_s,\mu_t)}{|s-t|}
    \end{equation*}
    exists for a.e. $t\in I$. Moreover, the map $t\mapsto|\mu'|(t)$ belongs to $L^2(I)$.
\end{proposition}
The map $t\mapsto |\mu'|(t)$ is called the metric derivative of $\boldsymbol{\mu}$. The following result from \citep{ambrosioGradientFlowsMetric2008} shows that it is minimal within the class of functions $m\in L^2([0,T];\R)$ satisfying~\eqref{eq: absolute continuity}.

\begin{proposition}
\label{prop: minimal velocity}
    Let $\boldsymbol{\mu}$ be a weakly continuous curve. 
    Then $\boldsymbol{\mu}$ belongs to $AC^2(I)$
    if and only if there exists a Borel map $\boldsymbol{v}:(y,t)\rightarrow v_t(y)\in \R^d$ such that $t\rightarrow \|v_t\|_{\mu_t}\in L^2(I)$ and the continuity equation
    \begin{equation*}
        \partial_t \mu_t +\nabla\cdot (\mu_t v_t)=0\quad \text{in } \R^d\times I
    \end{equation*}
    holds in the sense of distributions. 
    The map $\boldsymbol{v}$ is referred to as a velocity of the curve $\boldsymbol{\mu}$.
    There exists a unique velocity $\boldsymbol{v}^*$ with minimal $L^2$-norm, and it satisfies $\|v_t\|_{\mu_t}=|\mu'|(t)$ for a.e.\ $t\in I$.
\end{proposition}
On the metric space $(\mathcal P_2(E),\cW_2)$, the gradient flow can be understood in the sense of curves of maximal slope. We first recall the definition of upper gradient.

\begin{definition}
\label{def:upper gradient}
    For a given function $\phi: \mathcal{P}_2(E)\rightarrow\R$, a function $g:\mathcal P_2(E)\rightarrow [0,\infty]$ is a \emph{strong upper gradient} of $\phi$ if, for every curve $\boldsymbol{\mu}\in AC(I)$, the function $g\circ \boldsymbol{\mu}$ is Borel, and there holds
\begin{align*}
    |\phi \circ \boldsymbol{\mu}(t) - \phi \circ \boldsymbol{\mu}(s)| \le \int_s^t g \circ \boldsymbol{\mu}(r) |\boldsymbol{\mu}'|(r) \, dr, \quad \forall s, t \in I, \, s \le t.
\end{align*}
\end{definition}

To define curves of maximal slope, we take $I:=[0,\infty)$ from now on.
\begin{definition}
\label{def: maximal slope}
    Let $g$ be a strong upper gradient of $\phi$. A curve $\boldsymbol{\mu}\in AC^2_{loc}([0,\infty); \cP_2(E))$  is a \emph{curve of maximal slope} with respect to $g$ if   
\begin{align}
\label{eq: maximal slope}
-(\phi\circ\boldsymbol{\mu})'(t)=|\mu'|^2(t)=g^2(\mu_t),\quad \text{a.e. }t\in (0,\infty).
\end{align}
\end{definition}

Notice that if a curve $\boldsymbol{\mu}\in AC^2_{loc}(0,\infty)$ satisfies 
\begin{align*}
    \frac{1}{2}\int_0^t |\mu'|^2(s)\mathrm d s+\frac{1}{2}\int_0^t g^2(\mu_s)\mathrm d s +\phi(\mu_t)\leq \phi(\mu_0),\quad \text{a.e. } t\in (0,\infty),
\end{align*}
then $\boldsymbol{\mu}$ is a curve of maximal slope with respect to $g$ and satisfies~\eqref{eq: maximal slope}.

\begin{remark}
    For a given function $\phi:\mathcal P_2(E)\rightarrow \R$, the corresponding curves of maximal slope may not be unique.
\end{remark}

\begin{definition}[Metric slope on $(\cP_2(E),\cW_2)$]
\label{Def:slope}
Let $\phi:\cP_2(E)\to(-\infty,+\infty]$ be proper.
For $\mu\in\cP_2(E)$ with $\phi(\mu)<\infty$, the \emph{(descending) metric slope} of $\phi$ at $\mu$ is defined by
\begin{equation}
|\partial \phi|(\mu)
:=\limsup_{\cW_2(\nu,\mu)\to0}\frac{\big(\phi(\mu)-\phi(\nu)\big)^+}{\cW_2(\mu,\nu)},
\label{eq:metric-slope}
\end{equation}
where $(a)^+:=\max\{a,0\}$.
We set $|\partial \phi|(\mu)=+\infty$ if $\phi(\mu)=+\infty$.
\end{definition}

\begin{definition}[Strong Wasserstein subdifferential on $\cP_2(E)$]
\label{Def:Wsubdiff}
Let $\phi:\cP_2(E)\to(-\infty,+\infty]$ be proper and l.s.c.
Fix $\mu\in\cP_2(E)$ with $\phi(\mu)<\infty$.
We say that $\xi\in L^2(\mu;\R^d)$ belongs to the \emph{(strong) Wasserstein subdifferential} of $\phi$ at $\mu$,
and write $\xi\in\partial \phi(\mu)$, if there exists a function $o_\mu:[0,\infty)\to\R$ with
$o_\mu(r)/r\to0$ as $r\downarrow0$ such that, for every $\nu\in\cP_2(E)$ and every optimal coupling
$\gamma\in\Gamma_o(\mu,\nu)$,
\begin{equation}
\phi(\nu)-\phi(\mu)
\ \ge\
\int_{E\times E}\langle \xi(x),\,y-x\rangle\,\gamma(\d x,\d y)
\ -\ o_\mu\!\big(\cW_2(\mu,\nu)\big).
\label{eq:strong-wasserstein-subdiff}
\end{equation}
We set $\partial \phi(\mu)=\emptyset$ if $\phi(\mu)=+\infty$.
\end{definition}

\begin{definition}[Slope-realizing subgradients]
\label{Def:slope-realizing}
Define
\begin{equation}
\partial^\circ \phi(\mu)
:=\Big\{\eta\in\partial \phi(\mu):\ \|\eta\|_{L^2(\mu)}=|\partial \phi|(\mu)\Big\}\subset\partial \phi(\mu).
\label{eq:slope-realizing-subdiff}
\end{equation}
\end{definition}

\begin{remark}[Slope and subdifferential]
\label{Rem:slope-subdiff}
For $\xi\in\partial \phi(\mu)$, one typically has
$$
|\partial \phi|(\mu)\ \le\ \|\xi\|_{L^2(\mu)}.
$$
In particular, any $\xi\in\partial^\circ\phi(\mu)$ in the sense of \eqref{eq:slope-realizing-subdiff}
realizes the metric slope through its $L^2(\mu)$-norm.
\end{remark}

\subsection{McKean--Vlasov control problem}
For a given function $\phi\in \cC(\cP_2(E))$, to deal with the possible non-uniqueness of the gradient flows, we propose a control formulation for the potential functions. We first define general McKean--Vlasov control problems and state the relevant results. We distinguish between the diffusive regime ($\epsilon > 0$) and the deterministic transport regime ($\epsilon = 0$).

\subsubsection{Diffusive Regime (\texorpdfstring{$\epsilon > 0$}{epsilon > 0})}

 We fix a reference filtered probability space $(\Omega, \mathcal{F}, (\mathcal{F}_t)_{t \ge 0}, \mathbb{P})$ satisfying the usual conditions and supporting a $d$-dimensional Brownian motion $W$.

Let $\mathcal{U}_{ad}$ be the set of $\mathbb{F}$-adapted processes $\alpha: [0, \infty) \times \Omega \to \mathbb{R}^d$ such that for every $T>0$, $\alpha$ satisfies the square-integrability condition:
\begin{equation}
    \mathbb{E}^{\mathbb{P}_0}\left[ \int_0^T \|\alpha_t\|^2 \d t \right] < \infty,
\end{equation}
and the stochastic exponential
\begin{equation}
    \mathcal{E}_t(\alpha/\eps) := \exp\left( \int_0^t \frac{1}{\epsilon}\alpha_s \cdot dW_s - \frac{1}{2}\int_0^t \frac{1}{\epsilon^2}\|\alpha_s\|^2 \d s \right)
\end{equation}
is a uniformly integrable $\mathbb{P}_0$-martingale on $[0,T]$. For any control $\alpha \in \mathcal{U}_{ad}$, we define the probability measure $\mathbb{P}^\alpha$ on $\mathcal{F}_T$ via the Radon-Nikodym derivative:
\begin{equation}
    \frac{d\mathbb{P}^\alpha}{d\mathbb{P}_0} \bigg|_{\mathcal{F}_t} := \mathcal{E}_t(\alpha/\eps).
\end{equation}
By Girsanov's theorem, under $\mathbb{P}^\alpha$, the process $W_t^\alpha := W_t - \frac{1}{\epsilon}\int_0^t \alpha_s \d s$ is a Brownian motion. 
Consequently, the state process $(X_t)_{t\in [0,T]}$ (which acts as a scaled Brownian motion under $\mathbb{P}_0$) satisfies the controlled dynamics:
\begin{equation}
\label{eq:controlledSDE}
    \d X_t = \alpha_t \d t + \epsilon \,\d W_t^\alpha, \quad X_0 \sim \mu\in \cP_2(E).
\end{equation}
Denote by $\mu_t^\alpha := \text{Law}_{\mathbb{P}^\alpha}(X_t)$ the marginal distribution of the state at time $t$. Then it satisfies the Fokker-Planck equation in the sense of distributions
\begin{equation}
\label{eq: controlled_fokker_planck_eps}
    \partial_t \mu^\alpha_t=-\nabla\cdot(\alpha_t\mu_t^\alpha)+\frac{\eps^2}{2}\Delta\mu_t^\alpha. 
\end{equation}

Given a running cost $\ell: \mathcal{P}_2(E) \to \mathbb{R}$ which is continuous and bounded from below, the objective is to minimize over $\alpha\in \cU_{ad}$ the following cost function
\begin{equation} \label{eq:stochastic_cost}
    J_\eps(\mu, \alpha) := \mathbb{E}^{\mathbb{P}^\alpha} \left[ \int_0^\infty e^{-t} \left( \frac{1}{2}\|\alpha_t\|^2 + \ell(\mu_t^\alpha) \right) \d t \right],
\end{equation}
where $(X,\alpha)$ satisfies Equation~\eqref{eq:controlledSDE}.
While the problem is defined over path-dependent adapted controls, standard results in mean-field control allow us to restrict the search to Markovian feedback policies without loss of optimality.

\begin{definition}[Feedback controls]
    Define the set of admissible feedback controls as
    $$
    \mathcal A
:=\Big\{\alpha:[0,\infty)\times\mathbb R^d\to\mathbb R^d \ \Big|\
\alpha \text{ is Borel measurable, locally bounded, and of linear growth}\Big\}.
$$
\end{definition}

\begin{proposition}[Restriction to Feedback Controls] \label{prop:feedback_reduction}
Fix $\mu\in \cP_2(E)$. For every $\alpha\in \cA$, the associated SDE~\eqref{eq:controlledSDE} admits a weak solution. Moreover,
\begin{equation}
\label{eq:value function}
    v_\eps(\mu):=\inf_{\alpha \in \mathcal{U}_{ad}} J_\eps(\mu, \alpha) = \inf_{\alpha \in \mathcal{A}} J_\eps(\mu, \alpha),
\end{equation}
where the function $v_\eps$ is referred to as the value function.
\end{proposition}

\begin{proof}
See~\cite[Theorem 3.1]{PossamaControl} for the proof that the value function over open-loop controls coincides with the value function over feedback controls in the McKean--Vlasov setting.
\end{proof}

\subsubsection{Deterministic Regime (\texorpdfstring{$\epsilon = 0$}{epsilon=0})}

In the limit $\epsilon \to 0$, the probabilistic Girsanov formulation is replaced by the continuity equation formulation. The state is described directly by the curve of measures.

\begin{definition}[Admissible Pairs]
\label{def:admissiblepair}
    We say that a pair $(\boldsymbol\mu, \alpha)$ is admissible, denoted $(\boldsymbol\mu, \alpha) \in \mathcal{A}_0(\mu)$, if:
\begin{enumerate}
    \item $\boldsymbol\mu: [0, \infty) \to \mathcal{P}_2(E)$ is a continuous curve with $\boldsymbol\mu|_{t=0} = \mu$, i.e., $\mu_0=\mu$.
    \item $\alpha: [0, \infty) \times E \to \mathbb{R}^d$ is a Borel velocity field satisfying  $\int_0^\infty e^{-t} \|\alpha_t\|_{L^2(\mu_t)}^2 \d t < \infty$.
    \item The pair $(\boldsymbol\mu,\alpha)$ satisfies the Continuity Equation in the sense of distributions:
    \begin{equation} \label{eq:continuity_eq}
        \partial_t \mu_t + \nabla \cdot (\alpha_t \mu_t) = 0.
    \end{equation}
\end{enumerate}
\end{definition}
\begin{remark}
    For a curve
    $\boldsymbol{\mu}\in AC^2([0,T];\cP_2(E))$ with $(\boldsymbol{\mu},\alpha)\in \cA_0(\mu)$, we always identify the related (feedback) control $\alpha$ as the unique velocity of the curve with minimal $L^2$-norm via Proposition~\ref{prop: minimal velocity}.
\end{remark}

 Analogously, given the running cost function $\ell$, the cost function is defined as 
\begin{equation} \label{eq:deterministic_cost}
    J_0(\boldsymbol \mu, \alpha) := \int_0^\infty e^{-t} \left( \frac{1}{2} \int_{E} \|\alpha_t(x)\|^2 \mu_t(\d x) + \ell(\mu_t) \right) \d t,
\end{equation}
and, similar to above, the value function is defined as
\begin{equation}
\label{eq: det value function}
    v_0(\mu) = \inf_{(\boldsymbol\mu, \alpha) \in \mathcal{A}_0(\mu)} J_0(\boldsymbol{\mu}, \alpha).
\end{equation}

\begin{remark}
    Since we do not assume uniqueness of the solution to Equation~\eqref{eq:continuity_eq} with given control $\alpha$ and initial distribution $\mu$, the infimum in~\eqref{eq: det value function} is taken over all admissible pairs $(\boldsymbol{\mu},\alpha)\in \cA_0(\mu)$, compared to~\eqref{eq:value function}.
\end{remark}

\subsubsection{Dynamic Programming Principle}
It is well-known that the value function satisfies the dynamic programming principle. The proof of the following result can be found in \cite[Theorem 3.1, Theorem 3.2, Corollary 3.6]{DylanDPP}.

\begin{proposition}[Dynamic programming]\label{prop:DPP-early}
For every $T>0$ and $\mu\in\mc P_2(E)$,
\begin{itemize}
    \item Case $\eps>0$: \begin{equation}
    \label{eq:DPP_eps>0}
     v_\eps(\mu)=\inf_{\alpha\in \cA}\Big\{\int_0^{T} e^{-t}\Big(\tfrac12\,\E\|\alpha_t\|^2+\ell_\eps(\mu_t^\alpha)\Big)\d t+e^{-T}v_\eps(\mu_T^\alpha)\Big\}.
\end{equation}
\item Case $\eps=0$: \begin{equation}
\label{eq:DPP_eps=0}
     v_0(\mu)=\inf_{(\boldsymbol{\mu},\alpha)\in \mathcal A_0(\mu)}\Big\{\int_0^{T} e^{-t}\Big(\tfrac12\,\E\|\alpha_t\|^2+\ell_0(\mu_t^\alpha)\Big)\d t+e^{-T}v_0(\mu_T^\alpha)\Big\}.
\end{equation}
To unify the notation, we combine the previous two cases by introducing $$\cA_\eps(\mu):=\{(\boldsymbol{\mu},\alpha):\mu_0=\mu, \text{and } (\boldsymbol{\mu},\alpha) \text{ satisfies } \eqref{eq: controlled_fokker_planck_eps} \text{ if } \eps>0 \text{ resp. } \eqref{eq:continuity_eq} \text{ if } \eps=0\}.$$
Hence, for $\eps\ge 0$,
\begin{equation*}
    v_\eps(\mu)=\inf_{(\boldsymbol{\mu},\alpha)\in \cA_\eps(\mu)}\Big\{\int_0^{T} e^{-t}\Big(\tfrac12\,\E\|\alpha_t\|^2+\ell_\eps(\mu_t^\alpha)\Big)\d t+e^{-T}v_\eps(\mu_T^\alpha)\Big\}.
\end{equation*}
\end{itemize} 
\end{proposition}
From the dynamic programming principle, we can formally write down the stationary Hamilton-Jacobi-Bellman equation of the McKean--Vlasov control problem for $\eps\ge 0$:

\begin{equation}
\label{eq: HJB}
    w_\eps(\mu)=-\frac{1}{2}\int_E |\nabla\partial_\mu w_\eps(\mu)(x)|^2\mu(\mathrm d x)+\frac{\epsilon^2}{2}\int_E \Delta \partial_\mu w_\eps(\mu)(x)\mu(\d x)+\ell_\eps(\mu).
\end{equation}
As the HJB equation may not have classical solutions, weak solutions known as viscosity solutions are used to study the relation between the value function and the HJB equation. The definition of viscosity solution can be found in Appendix~\ref{ex: gfnotminimizer}. We also recall some results on viscosity solutions from ~\citep{MetricvissolGanbo,AmbFen14, Cosso2022MasterBE,SonerYan}.

\begin{proposition}
\label{prop:vfHJB}
The value function defined in ~\eqref{eq:value function} (resp.~\eqref{eq: det value function}) is a viscosity solution of the HJB equation~\eqref{eq: HJB} for $\eps>0$ (resp.~$\eps=0$).
\end{proposition}

Now we specify the class of functions we consider. Denote $\cG_2(\mu):=1+\int_E \|x\|^2\mu(\d x)$. Define the set of functions with at most "quadratic" growth by
\begin{equation}
\label{eq: quadratic growth function}
    \mathcal C_{\mathrm{quad}}
:=\Big\{v:\mathcal P_2\to\mathbb R\;\Big|\; v\text{ is continuous, bounded from below, and }
\exists\, C>0:\;|v(\mu)|\le C\,\mathcal G_2(\mu),\ \forall\mu\Big\}.
\end{equation}

\begin{theorem}
\label{thm: uniqueness of HJB}
    For every $\epsilon\ge 0$, ~\eqref{eq: HJB} admits at most one  viscosity solution in $\mathcal C_{\mathrm{quad}}$.
\end{theorem}
\begin{proof}
    When $\epsilon>0$, the proof idea of \cite[Theorem 4.1]{SonerYan} in finite-time setting can be easily adapted to the discounted infinite-horizon setting here. When $\epsilon=0$, the theorem is stated in~\cite[Section 5, Theorem 5.1 and Theorem 5.2]{MetricvissolGanbo}.
\end{proof}
Combining the previous two results, we can characterize the optimal value function~\eqref{eq:value function}  (resp. \eqref{eq: det value function}) 
as the unique viscosity solution of the HJB equation~\eqref{eq: HJB} for $\eps>0$ (resp.~$\eps=0$).

\subsection{From McKean--Vlasov control to gradient flow}
\label{subsec:gfisminimizer}
Now we show that, for every given functional $V\in \cC^2(\cP_2)$ and every $\eps\ge 0$, we can choose a suitable running cost function $\ell_\eps$ such that the optimal value function in~\eqref{eq:value function} (resp. \eqref{eq: det value function}) coincides with $V$ as function on $\cP_2(E)$. This can be regarded as an inverse problem. The connection between McKean--Vlasov control and Langevin dynamic is also mentioned in~\cite{hofer2025optimalcontrolpotentialgames}.

\begin{theorem}[Identification of $\ell$]\label{thm:controlformulation}
Fix $\eps> 0$ (resp. $\eps=0$). For any given function $V\in \cC^2(\cP_2)\cap \cC_{quad}$ (resp. $V\in \cC^1(\cP_2)\cap \cC_{quad}$), define the running cost function $\ell_\eps: \cP_2(E)\rightarrow \R$ by 
\begin{align}
\label{eq:ell}
&\ell_\eps(\mu):=V(\mu)+\tfrac12\!\int_E\!\|\nabla_x\partial_\mu V(\mu)\|^2\,\d\mu\;-\;
 \tfrac{\varepsilon^2}{2}\!\int_E\!\Delta_x\partial_\mu V(\mu)\,\d\mu\\
 \nonumber &(resp. \quad \ell_0(\mu):=V(\mu)+\tfrac12\!\int_E\!\|\nabla_x\partial_\mu V(\mu)\|^2\,\d\mu\,).
\end{align}
Let $v_\epsilon$ be the value function defined in~\eqref{eq:value function}  (resp. \eqref{eq: det value function}) with running cost $\ell_\eps$. Then  $v_\eps\equiv V$. 
\end{theorem}
\begin{proof}
    From Proposition~\ref{prop:vfHJB} with running cost $\ell_\eps$, the value function $v_{\eps}$ defined in~\eqref{eq:value function}  (resp. \eqref{eq: det value function}) with running cost $\ell_\eps$ is a viscosity solution of the HJB equation~\eqref{eq: HJB} and belongs to $\cC_{quad}$.
    Moreover, from the definition of $\ell_\eps$  in \eqref{eq:ell}, the given function $V$  also satisfies the same HJB equation, and $V\in \cC_{quad}$. By Theorem~\ref{thm: uniqueness of HJB}, there is a unique viscosity solution of the HJB equation in $\cC_{quad}$, so $v_{\eps}\equiv V$.
\end{proof}
\begin{remark}
    When $\eps=0$, we only need $V\in \cC^1(\cP_2)$ to deduce the same statement as in the previous theorem.
\end{remark}

We first discuss the properties of the optimal flows from the McKean--Vlasov control problem  when $\eps=0$, showing that \emph{every} minimizing flow is a curve of maximal slope.
\begin{proposition}[Every minimizer is a curve of maximal slope for $v$]
\label{prop: minimizer is ms}
Let $\eps=0$, and denote the value function $v_0$ by $v$. Assume that the value function $v_\eps:\mathcal P_2\to\mathbb R$ in ~\eqref{eq:value function} (resp. \eqref{eq: det value function}) is proper, lower semicontinuous, and $\cW_2$--locally Lipschitz on its sublevel sets.
Then its metric slope $|\partial v|$ is a strong upper gradient on $\{v<+\infty\}$. Let $(\boldsymbol{\mu}^*,\alpha^*)$ be a minimizer of $J_0$  among all admissible pairs with initial condition $\mu$.
Then the following statements hold: 
\begin{enumerate}
\item $\boldsymbol{\mu}^*$ is absolutely continuous in $(\mathcal P_2,\cW_2)$ with metric speed 
$|(\mu^*)'|\in L^2_{\mathrm{loc}}([0,\infty))$, and
\begin{equation}\label{eq:E1}
\int_{E} |\alpha_t^*|^2\,\d\mu_t^* = |(\mu^*)'|_t^{\,2}\qquad\text{ for a.e.\ $t>0$}.
\end{equation}
\item $\boldsymbol\mu^*$ satisfies the energy--dissipation equality
\begin{equation}\label{eq:E2}
-\frac{\d}{\d t}v(\mu_t^*)=\tfrac12\,|\partial v|(\mu_t^*)^{2}
+\tfrac12\,|(\mu^*)'|_t^{\,2}\qquad\text{for a.e. }t>0.
\end{equation}
In particular, $\boldsymbol{\mu}^*$ is a curve of maximal slope for $v$ w.r.t.\ the strong upper gradient $|\partial v|$.
\item There exists a measurable selection $\zeta_t\in \partial^\circ v(\mu_t^*)$ (the metric subdifferential 
with minimal $L^2(\mu_t^*)$--norm) such that
\begin{equation}\label{eq:E3}
\|\zeta_t\|_{L^2(\mu_t^*)}=|\partial v|(\mu_t^*),
\qquad
\alpha_t^*=-\zeta_t\quad\text{for a.e. }t>0. 
\end{equation}
\end{enumerate}
Hence every minimizer $(\boldsymbol{\mu}^*,\alpha^*)$ evolves according to a (metric) gradient--flow dynamics for $v$.
\end{proposition}

\begin{proof}
Let $(\boldsymbol{\mu},\alpha)\in\cA_0(\mu)$ be any admissible pair. By the definition of metric derivative,
\begin{equation}\label{eq:BB}
|\mu'|(t) \le \Big(\int_{E} |\alpha_t|^2\,\d\mu_t\Big)^{1/2}\quad\text{for a.e. }t>0.
\end{equation}
Since $|\partial v|$ is a strong upper gradient, we have, along $\boldsymbol\mu$,
$$
-\frac{\d}{\d t}v(\mu_t)\le |\partial v|(\mu_t)\,|\mu'|(t)
\le \tfrac12\,|\partial v|(\mu_t)^2+\tfrac12\,|\mu'|(t)^{2}
\le \tfrac12\,|\partial v|(\mu_t)^2+\tfrac12\!\int_{E} |\alpha_t|^2\,\d\mu_t,
$$
for a.e.\ $t>0$. Using \eqref{eq: HJB} with $\eps=0$, we have $\ell(\mu_t)= v(\mu_t)+\tfrac12|\partial v|(\mu_t)^2$,
hence for a.e.\ $t>0$,
\begin{equation}\label{eq:ineq}
\int_E \frac{1}{2}|\alpha_t|^2\,\d\mu_t+\ell(\mu_t)
\ge \,v(\mu_t)-\frac{\d}{\d t}v(\mu_t).
\end{equation}
Multiplying \eqref{eq:ineq} by $e^{-t}$ and integrating over $[0,T]$ yields
$$
\int_0^T e^{-t}\!\Big(\!\int_E \frac{1}{2}|\alpha_t|^2\,\d\mu_t+\ell(\mu_t)\Big)\,\d t
\ge \int_0^T e^{-t}\Big( v(\mu_t)-\tfrac{\d}{\d t}v(\mu_t)\Big)\,\d t.
$$
Integration by parts on the right-hand side gives
$$
\int_0^T e^{-t}\!\Big(\!\int_{E} \frac{1}{2}|\alpha_t|^2\,\d\mu_t+\ell(\mu_t)\Big)\,\d t
\ge -e^{-T}v(\mu_T)+v(\mu).
$$
Letting $T\to\infty$, from DPP~\eqref{eq:DPP_eps=0} we have $e^{-T}v(\mu_T)\to0$, and thus  
$$
J(\boldsymbol\mu,\alpha)\ge v(\mu).
$$
If $(\boldsymbol \mu^*,\alpha^*)$ is optimal, equality holds everywhere above. 
Equality in \eqref{eq:BB} yields \eqref{eq:E1}, while equality in Young's inequality
forces the existence of $\zeta_t$ realizing the slope with 
$\|\zeta_t\|_{L^2(\mu_t^*)}=|\partial v|(\mu_t^*)$
and $\alpha_t^*=-\zeta_t$, giving \eqref{eq:E3}.
Plugging this into \eqref{eq:ineq} yields \eqref{eq:E2}.
By definition, \eqref{eq:E2} is the energy--dissipation equality, 
so $\boldsymbol\mu^*$ is a curve of maximal slope for $v$ with respect to $|\partial v|$.
\end{proof}

\subsection{From gradient flow to McKean--Vlasov control}
In Section~\ref{subsec:gfisminimizer}, we showed that every minimizer of the McKean--Vlasov control problem is a curve of maximal slope of the value function. In this section, we show that,
 with sufficient assumptions on the given potential functions,  every curve of maximal slope~\eqref{eq: maximal slope} is a minimizing flow of~\eqref{eq:value function}.

 \begin{remark}
     Notice that this inverse direction is not always true. We provide a counterexample in Appendix~\ref{ex: gfnotminimizer}.
 \end{remark}

 We start with the definition of Calibration Property. This is not a new mathematical concept, but collects in a single definition the well-known conditions that: (i) a curve of maximal slope satisfies the energy–dissipation equality, (ii) its metric derivative is represented by a velocity field in the continuity equation, and (iii) the slope is realized by a minimal-norm element of the metric subdifferential. These concepts exist separately in \cite{ambrosioGradientFlowsMetric2008}, \cite{BenamouBrenier2000}, and \cite{Carmona2018Probabilistic}.
We use the term Calibration Property purely for convenience.

\begin{definition}[Calibration Property \textnormal{(Cal)}]
\label{def:Cal}
Let $ v:\cP_2(E)\to(-\infty,+\infty]$ be proper, l.s.c., and $\cW_2$--locally Lipschitz on its sublevel sets.
We say that $v$ satisfies the \emph{Calibration Property} \textnormal{(Cal)} if for every curve $(\mu_t)_{t\ge0}$ of maximal slope for $v$ with respect to $|\partial v|$, 
there exist Borel measurable maps
$$
\alpha : [0,\infty) \times E \to \R^d,
\qquad
\zeta \in L^2(\mu_t \otimes dt;\R^d),
$$ 
with $\zeta_t := \zeta(t,\cdot)\in \partial^\circ v(\mu_t)$ for a.e.\ $t\ge0$, such that for a.e.\ $t\ge0$:
\begin{align*}
\text{\emph{(CE)}}\quad
&\text{$(\boldsymbol{\mu},\alpha)$ solves the continuity equation~\eqref{eq:continuity_eq} in the weak sense,}\\
&|\mu'|^2(t)=\int_{E}|\alpha_t(x)|^2\,\mu_t(\d x);\\[0.2cm]
\text{\emph{(CR)}}\quad
&\frac{\d}{\d t}v(\mu_t)=\int_{E}\zeta_t(x)\cdot\alpha_t(x)\,\mu_t(\d x),
\qquad
\|\zeta_t\|_{L^2(\mu_t)}=|\partial v|(\mu_t).
\end{align*}
\end{definition}

\begin{remark}
    Intuitively, property \textnormal{(Cal)} asserts that along any maximal slope curve there exists a \emph{minimal velocity}
$\alpha$ (realizing the metric derivative) and a \emph{slope-realizing} subgradient $\zeta$
with $\|\zeta_t\|_{L^2(\mu_t)}=|\partial v|(\mu_t)$, such that the chain rule holds with equality.
\end{remark}

\begin{remark}[Sufficient conditions for \textnormal{(Cal)}]
Each of the following guarantees the Calibration Property \textnormal{(Cal)}.
\begin{itemize}
    \item[a.] Smooth $v$ (Lions $C^1$ + chain rule).\\
Assume $v\in \cC^1$ on $\cP_2(E)$ and its Lions derivative
$$
\nabla_W v(\mu):=\nabla_x\partial_\mu v(\mu)\in L^2(\mu;\R^d)
$$
is well-defined.
If for every absolutely continuous curve $\boldsymbol{\mu}$ with velocity $\alpha$ solving the continuity equation,
the chain rule holds:
$$
\frac{\d}{\d t}v(\mu_t)=\int_{\R^d}\nabla_W v(\mu_t)(x)\cdot \alpha_t(x)\,\mu_t(\d x),
$$
then along a maximal slope curve one may take $\zeta_t=\nabla_W v(\mu_t)$ to obtain \textnormal{(CR)};
and by Proposition~\ref{prop: minimal velocity}, one can choose a  velocity $\alpha=-\zeta$ so that it is the minimal velocity such that \textnormal{(CE)} holds. 
\item[b.] Convex subdifferential with measurable minimal-norm selector.\\
Assume that for every $\mu$ with $v(\mu)<\infty$, the set $\partial v(\mu)\subset L^2(\mu;\R^d)$ is a nonempty closed convex set 
and admits a measurable selection $\zeta^{\min}(\mu)\in\partial v(\mu)$ satisfying
$$
\|\zeta^{\min}(\mu)\|_{L^2(\mu)}=|\partial v|(\mu).
$$
Then along any maximal slope curve, the energy--dissipation equality implies equality in Young's inequality,
yielding \textnormal{(CR)} with $\zeta_t=\zeta^{\min}(\mu_t)$; \textnormal{(CE)} again follows by choosing a minimal velocity
via Proposition~\ref{prop: minimal velocity}.
\end{itemize}
\end{remark}

\begin{proposition}[Energy identity under \textnormal{(Cal)}]
\label{prop:cal-energy}
Let $\phi:\cP_2(E)\to(-\infty,+\infty]$ satisfy \textnormal{(Cal)}
in the sense of Definition~\ref{def:Cal}. 
Let $\boldsymbol{\mu}:=(\mu_t)_{t\ge 0}$ be a curve of maximal slope for $\phi$ with respect to $|\partial \phi|$.
Let $(\alpha,\zeta)$ be the measurable selections provided by \textnormal{(Cal)}.
Then, for a.e.\ $t\ge 0$,
$$
\alpha_t=-\zeta_t \qquad \mu_t\text{-a.e.}, 
$$
and
$$
-\frac{\d}{\d t}\phi(\mu_t)=\int_{\R^d}|\alpha_t(x)|^2\,\mu_t(\d x)
=\int_{\R^d}|\zeta_t(x)|^2\,\mu_t(\d x).
$$
Consequently, for every $T>0$,
$$
\phi(\mu_0)-\phi(\mu_T)=\int_0^T\int_{\R^d}|\alpha_t(x)|^2\,\mu_t(\d x)\,\d t.
$$
\end{proposition}

\begin{proof}
Since $\boldsymbol{\mu}$ is a curve of maximal slope for $\phi$ with respect to $|\partial \phi|$, the energy--dissipation
equality holds:
$$
\phi(\mu_0)-\phi(\mu_T)
=\frac12\int_0^T |\mu'|(t)^2\,\d t+\frac12\int_0^T|\partial \phi|(\mu_t)^2\,\d t.
$$
In particular, $t\mapsto \phi(\mu_t)$ is absolutely continuous and, for a.e.\ $t\in(0,T)$,
$$
-\frac{\d}{\d t}\phi(\mu_t)=|\mu'|(t)^2=|\partial \phi|(\mu_t)^2.
$$

On the other hand, by \textnormal{(CR)} and Cauchy--Schwarz,
$$
-\frac{\d}{\d t}\phi(\mu_t)=\int_{E}\zeta_t(x)\cdot\alpha_t(x)\,\mu_t(\d x)
\le \|\zeta_t\|_{L^2(\mu_t)}\,\|\alpha_t\|_{L^2(\mu_t)}.
$$
Using \textnormal{(CR)} and \textnormal{(CE)} we have
$\|\zeta_t\|_{L^2(\mu_t)}=|\partial \phi|(\mu_t)$ and $\|\alpha_t\|_{L^2(\mu_t)}=|\mu'|(t)$, hence
$$
-\frac{\d}{\d t}\phi(\mu_t)\le |\partial \phi|(\mu_t)\,|\mu'|(t).
$$
Combining with the pointwise identity $-\frac{\d}{\d t}\phi(\mu_t)=|\mu'|(t)^2=|\partial \phi|(\mu_t)^2$
yields equality in the above inequality. Therefore equality holds in Cauchy--Schwarz, which implies
$\alpha_t=-\zeta_t$ $\mu_t$-a.e.\ (on the set where the common norm is nonzero; the claim is trivial otherwise).
Plugging $\alpha_t=-\zeta_t$ back into \textnormal{(CR)} gives
$$
-\frac{\d}{\d t}\phi(\mu_t)=\int_{E}|\alpha_t(x)|^2\,\mu_t(\d x)
=\int_{E}|\zeta_t(x)|^2\,\mu_t(\d x).
$$
Integrating from $0$ to $T$ yields
$$
\phi(\mu_0)-\phi(\mu_T)=\int_0^T\int_{E}|\alpha_t(x)|^2\,\mu_t(\d x)\,\d t.
$$
\end{proof}

\begin{theorem}[Every gradient flow is optimal solution]
Suppose $v$ satisfies Property \textnormal{(Cal)} in Definition~\ref{def:Cal} with its metric slope $|\partial v|$ being a strong upper gradient on set $\{\mu\in \cP_2(E): v(\mu)<\infty\}$. Assume $v$ satisfies the HJB equation~\eqref{eq: HJB}.
Then, for any curve of maximal slope
$(\mu_t)_{t\ge0}$ with initial distribution $\mu$ ,
the pair $(\boldsymbol{\mu},\alpha^\star)$  with $\alpha_t^\star:=-\zeta_t$ (given by \textnormal{(Cal)}) is an optimal solution of
\begin{equation*}
    \inf_{(\boldsymbol{\tilde \mu},\tilde\alpha)\in \cA_0(\mu)}\;
\int_0^\infty e^{-t}\!\Big(\int_E \tfrac12|\tilde\alpha_t(x)|^2\,\tilde\mu_t(\d x) + \ell(\tilde\mu_t)\Big)\,\d t,
\qquad
\partial_t\tilde\mu_t+\nabla\!\cdot(\tilde\mu_t\tilde\alpha_t)=0,\ \tilde\mu_0=\mu,
\end{equation*}
and all such optimal pairs achieve the same value $J=v(\mu_0)$.
\end{theorem}

\begin{proof}
Let $\boldsymbol\mu=(\mu_t)_{t\ge0}$ be a curve of maximal slope and for every $t\ge 0$, choose $(\zeta_t,\alpha_t)$ from \textnormal{(Cal)}.
By Young’s inequality and \textnormal{(CR)}, 
\[
-\frac{\d }{\d t} v(\mu_t)=\int_E \zeta_t\!\cdot\!\alpha_t\,\d\mu_t
\le \tfrac12\|\zeta_t\|_{L^2(\mu_t)}^2+\tfrac12\!\int_E|\alpha_t|^2\,\d\mu_t
=\tfrac12|\partial v|(\mu_t)^2+\tfrac12|\mu'|(t)^2,
\]
with equality since $\boldsymbol \mu$ is a curve of maximal slope. Hence $\alpha_t=-\zeta_t$ a.e.\ and
$\tfrac12\int_E |\alpha_t|^2\,\d\mu_t=\tfrac12\,|\partial v|(\mu_t)^2$.
Using HJB equation~\eqref{eq: HJB} with $\eps=0$,
$$
\int_E |\alpha_t|^2\,\d\mu_t+\ell(\mu_t)= v(\mu_t)-\partial_t v(\mu_t).
$$
Multiplying by $e^{-t}$ and integrating on $[0,\infty)$, the discounted boundary term vanishes and the cost equals $v(\mu)$.
Thus $(\boldsymbol\mu,\alpha^\star)$ is optimal and the value is unique.
\end{proof}

\section{Early Stopping Formulation}
\label{sec:3}
\subsection{An endpoint formulation from DPP}
 For every $\eps\ge 0$,  $T>0$ and $\mu\in \cP_2(E)$, we denote the reachable sets by  
 \begin{equation}
 \label{eq: reachable set}
     \cT_{T,\eps}(\mu):=\begin{cases}
         \{\nu\in \cP_2(E) : \text{exists}\,(\alpha, \,\boldsymbol{\mu}^{\alpha}) \text{ satisfying} ~\eqref{eq:controlledSDE}, \text{ s.t. } \mu^{\alpha}_T=\nu\},\qquad \eps>0,\\
         \{\nu\in \cP_2(E) : \text{exists}\,(\alpha, \,\boldsymbol{\mu}^{\alpha}) \text{ satisfying} ~\eqref{eq:continuity_eq}, \text{ s.t. } \mu^{\alpha}_T=\nu\},\qquad \eps=0.
     \end{cases}
 \end{equation}
When $\eps=0$, we may omit the dependence on $\eps$, and write $\cT_T(\mu)$ for $\cT_{T,0}(\mu)$.
\begin{remark}
\label{rmk:reachable set}
Characterizations of the reachable sets are known in the literature, from either the Schrödinger bridge theory or the optimal transport theory. We recall some results here.
    \begin{itemize}
        \item  When $\eps>0$, for every $\mu,\nu$ such that $\cH(\mu|\ \textit{Leb})<\infty$ and $\cH(\nu|\ \textit{Leb})<\infty$,  there exists $\alpha\in \cA$ such that $\nu\in \cT_{T,\eps}(\mu)$ for every $T>0,\eps>0$. See ~\cite[Theorem 2.12]{Leonard2014SchrodingerSurvey}.
        \item When $\eps=0$, for every given $\mu\in \cP_2(E)$ and every $\nu\in \cP_2(E)$, there exists $\alpha\in \cA$ such that $\nu\in \cT_{T,0}(\mu)$. Moreover, there exists $\alpha^*\in \cA$ such that $\frac{1}{2}\int_0^T\int_E |\alpha_t|^2\mathrm d \mu^{\alpha}_t\mathrm \d t=\frac{1}{2T}\cW_2^2(\mu,\nu)$. See~\cite[Theorem 5.27]{Santambrogio2015OTAM}.
    \end{itemize}
\end{remark}
For every $\mu,\nu\in \cP_2(E)$, we denote the pair of the curve and the corresponding controls that transport $\mu$ to $\nu$ by
\begin{equation*}
\label{eq:controls_two_marginals}
\cA_{T,\eps}(\mu,\nu):=\begin{cases}
    \{(\boldsymbol{\mu},\alpha)\in \cA_\eps(\mu): \exists\,\boldsymbol{\mu}\in AC^2([0,T]) \textit{ s.t. }(\boldsymbol{\mu},\alpha) \text{ satisfies}~\eqref{eq:controlledSDE}, \, \mu_0=\mu,\, \mu_T=\nu\},\quad \eps>0,\\
    \{(\boldsymbol{\mu},\alpha)\in \cA_\eps(\mu): \exists\,\boldsymbol{\mu}\in AC^2([0,T]) \textit{ s.t. }(\boldsymbol{\mu},\alpha) \text{ satisfies}~\eqref{eq:continuity_eq}, \, \mu_0=\mu,\, \mu_T=\nu\},\quad\eps=0.
\end{cases}
\end{equation*}
If $\nu\notin\cT_{T,\eps}(\mu)$, then $\cA_{T,\eps}(\mu,\nu)=\emptyset$. For $\nu\in \cT_{T,\eps}(\mu)$, we define the \emph{discounted energy} as 
$$
d_{T,\eps}^2(\mu,\nu):=\inf\Big\{\int_0^{T} e^{-t}\Big(\tfrac12\,\E\|\alpha_t\|^2+\ell_{\eps}(\mu_t)\Big)\d t:\  \partial_t\mu_t = -\nabla_x\cdot(\mu_t\alpha_t)+\tfrac{\varepsilon^2}{2}\Delta_x\mu_t, \mu_0=\mu, \mu_T=\nu\ \Big\}.
$$

We recall from Proposition~\ref{prop:DPP-early} that the optimal value  function $v$ of~\eqref{eq:value function} satisfies the dynamic programming principle~\ref{prop:DPP-early}. Next we show an equivalent endpoint formulation from the dynamic programming principle.

\begin{proposition}[Endpoint formulation from DPP]\label{prop:endpoint-DPP}
For every $T>0$, $\mu\in \cP_2(E)$ and $\eps\ge 0$,
\begin{align}
\label{eq:endpoint-discount}
     v_\eps(\mu)=\inf_{\nu\in\mc T_{ T,\eps}(\mu)}\Big\{\,d_{T,\eps}^2(\mu,\nu)+e^{-T}v_\eps(\nu)\,\Big\}. 
\end{align}
\end{proposition}

\begin{proof}
Fix $T>0$ and $\mu\in\mathcal P_2(E)$. By Proposition~\ref{prop:DPP-early},
$$
v_\eps(\mu)
=\inf_{(\boldsymbol{\mu},\alpha)\in \cA_{T,\eps}(\mu)}\Big\{\int_0^T e^{-t}\Big(\tfrac12\,\E\|\alpha_t\|^2+\ell(\mu_t)\Big)\,\d t
+ e^{-T}v_\eps(\mu_T)\Big\}.
$$  
For any admissible pair $(\boldsymbol{\mu},\alpha)\in \cA_\eps(\mu)$,  set $\nu:=\mu_T\in\mc T_{T,\eps}(\mu_0)$. Then, by definition,
$$
\int_0^T e^{-t}\Big(\tfrac12\,\E\|\alpha_t\|^2+\ell(\mu_t)\Big)\,\d t
\;\ge\; d_{T,\eps}^2(\mu,\nu).
$$
Adding $e^{-T}v_\eps(\nu)$ on both sides, and taking the infimum over $\cA_\eps(\mu)$, yields
$$
v_\eps(\mu)\;\ge\; \inf_{\nu\in\cT_{T,\eps}(\mu)}\Big\{d_{T,\eps}^2(\mu_0,\nu)+e^{-T}v_\eps(\nu)\Big\}. 
$$

Conversely, let $\nu\in\mc T_{T,\eps}(\mu)$  be arbitrary. By definition of the reachable set, there exists an admissible pair $(\boldsymbol\mu,\alpha)\in \cA_\eps(\mu)$  on $[0,T]$ with initial distribution $\mu$  and with $\mu_T=\nu$. Plugging this pair into the dynamic programming identity gives
$$
v_\eps(\mu)\;\le\;\int_0^T e^{-t}\Big(\tfrac12\,\E\|\alpha_t\|^2+\ell(\mu_t)\Big)\,\d t
+ e^{-T}v_\eps(\nu).
$$
Taking the infimum over all admissible $(\boldsymbol\mu,\alpha)$ with the same endpoints $(\mu,\nu)\in \cA_\eps(\mu)$ yields
$$
v_\eps(\mu)\;\le\; d_{T,\eps}^2(\mu,\nu)+e^{-T}v_\eps(\nu).
$$
Since $\nu\in\mc T_{T,\eps}(\mu)$  is arbitrary, we obtain
\[
v_\eps(\mu_0)=\inf_{\nu\in\mc T_{T,\eps}(\mu)}\Big\{d_{T,\eps}^2(\mu,\nu)+e^{-T}v_\eps(\nu)\Big\}. 
\]
This is the desired endpoint formulation.
\end{proof}

\begin{theorem}[\texorpdfstring{$\Gamma$}{Gamma}-convergence with endpoint formulations]\label{thm:gamma}
Fix $V\in \cC_{quad}$ and $\mu\in \cP_2(E)$. For every $T>0$, as $\varepsilon\to0$, the functional $\nu \mapsto d_{T,\varepsilon}^2(\mu,\nu)+e^{-T}V(\nu)$ $\Gamma$-converges (in $\cW_2$) to $\nu \mapsto d_{T,0}^2(\mu,\nu)+e^{-T}V(\nu)$. In particular, minimizers $\nu_{\varepsilon,T}$ converge (along subsequences) to minimizers of the limit functional.
\end{theorem}
The proof is postponed to Appendix~\ref{apd:gamma}.
\subsection{An endpoint formulation from curves of maximal slope}
In this section, we focus on the deterministic case and set $\eps=0$, and omit the dependence on $\eps$. Recall from Definition~\ref{def:upper gradient} and Definition~\ref{def: maximal slope} the concepts of upper gradient and curve of maximal slope. Recall that in Proposition~\ref{prop: minimizer is ms}, we show that 
\begin{equation}
\label{eq:G}
G(\mu):=\sqrt{2\left(\ell(\mu)-v_0(\mu)\right)} \ge 0
\end{equation}
is an upper gradient of $v$. For every $\mu\in \cP_2(E)$  and $\nu\in \cT_T(\mu)$, we define the \emph{Finsler metric} 
\begin{align}
\label{eq:finslerdistance}
d_G(\mu,\nu)&:=\inf_{\boldsymbol{\mu}\in \cA_{1,0}(\mu,\nu)}\left\{\int_0^1  |\mu'|(t)G(\mu_t)\,\mathrm d t\right\},
\end{align}
and the energy
\begin{align}
    \label{eq: finsler_energy}
    \mathcal E_G^T(\mu,\nu)
:=
\inf\left\{
\int_0^T\Big(\tfrac12|\eta'|^2(t)+\tfrac12G(\eta_t)^2\Big)\,\d t:\ 
\eta\in AC_2([0,T];\mathcal P_2(E)),\ \eta_0=\mu,\ \eta_T=\nu
\right\}.
\end{align}
\begin{remark}
    This action energy is different from the previous discounted energy with different running cost function. There is no discount factor. Moreover, $d_G$ does not depend on the choice of integrated time horizon.
\end{remark}
We first restate some standard results of the time and length reparameterization of curves. The following lemma can be found in~\cite{RossiSavareMielke2008}. 
\begin{lemma}[Length and time reparameterization]
\label{lem:reparametrization}
    Fix $0\le s<t$. Let $\omega:\cP_2(E)\to [0,\infty)$ be a Borel function, and let $\boldsymbol{\mu}\in AC([s,t]; \cP_2(E))$. If $L:=\int_s^t \frac{|\mu'|(r)}{\omega(\mu_r)}\d r<\infty$, then the reparametrized curve $\boldsymbol{\mu}_{\omega}:[0,L]\to \cP_2(E)$ with
    \begin{align*}      \mu_\omega(r):=\mu(\kappa_{\omega}(r)),\qquad \kappa_{\omega}(r):=\inf\left\{p\in [s,t]: \int_s^p \frac{|\mu'|(\tau)}{\omega(\mu(\tau))}\d \tau=r\right\},\qquad r\in [0,L],
    \end{align*}
    satisfies
    \[
    \boldsymbol{\mu}_{\omega}\in AC([0,L];\cP_2(E)),\quad \mu_\omega(0)=\mu_s,\quad \mu_\omega(L)=\mu_t,\quad \omega(\mu_{\omega}(r))=|\mu'_\omega|(r) \quad \textit{a.e. } r\in [0,L].
    \]
    If additionally $\int_s^t |\mu'|(r)\omega(\mu_r) \, \d r < \infty$, we have $\boldsymbol{\mu}_{\omega}\in AC^2([0,L];\cP_2(E))$ and
    \[
        \int_0^L |\mu_\omega'|^2(r) \, \d r = \int_0^L \omega(\mu_\omega(r))^2 \, \d r = \int_s^t |\mu'|(r) \omega(\mu_r) \, \d r < \infty.
    \]
\end{lemma}

\begin{proposition}[Basic properties of $d_G$]\label{prop:dg-basic}

For every $\mu,\nu\in \cP_2(E)$, we have the following:
\begin{itemize}
    \item[(i)] $d_G(\mu,\nu)\ge0$ and $d_G(\mu,\nu)=d_G(\nu,\mu)$.
    \item[(ii)] $d_G(\mu^0,\mu^1)+d_G(\mu^1,\mu^2)\ge d_G(\mu^0,\mu^2)\hspace*{0.4cm} \forall\,\mu^0,\mu^1,\mu^2\in \cP_2(E)$.
    \item[(iii)] $d_G(\mu,\nu)=\inf_{T>0} \cE_G^T(\mu,\nu)$.
\end{itemize}  
In general, $d_G$ is a pseudo-distance on $\cP_2(E)$. 
\end{proposition}
The proof is postponed to Appendix~\ref{apx:proofs}. Similar to Proposition~\ref{prop:endpoint-DPP}, curves of maximal slope provide another form of endpoint formulation. Recall that $G$  is a strong upper gradient of $v_0$.

\begin{proposition}[Endpoint formulation from curves of maximal slope]
\label{prop:endpoint_MS}
Fix $\mu\in \cP_2(E)$. Assume there exists a curve of maximal slope
$\eta\in AC^2_{\mathrm{loc}}([0,\infty);\mathcal P_2(E))$
for $v_0$ with respect to $G$ such that $\eta_0=\mu$, then for every $T>0$,
\begin{align}
\label{eq:perT_argmin}
    v_0(\mu)=\inf_{\nu\in \cT_T(\mu)}\left\{ \cE^T_G(\mu,\nu)+v_0(\nu)\right\}.
\end{align}
Moreover, the infimum is attained at $\nu=\eta_T$.
\end{proposition}
\begin{proof}

Fix $\nu\in\mathcal P_2(E)$ and any curve $\rho\in AC_2([0,T])$ with
$\rho_0=\mu$ and $\rho_T=\nu$.
Since $G$ is a strong upper gradient, we have
$$
v_0(\mu)-v_0(\nu)\ \le\ \int_0^T G(\rho_t)\,|\rho'|(t)\,\d t\le
\int_0^T\Big(\tfrac12|\rho'|^2(t)+\tfrac12G(\rho_t)^2\Big)\,\d t.
$$
Taking the infimum over all such $\eta$ yields
\begin{equation}\label{eq:universal_bound}
v_0(\mu)\ \le\ v_0(\nu)+\mathcal E_G^T(\mu,\nu),
\qquad \forall\, \nu\in\mathcal P_2(E).
\end{equation}
In particular,
\begin{equation}\label{eq:inf_lowerbd}
v_0(\mu)\ \le\ \inf_{\nu}\Big\{v_0(\nu)+\mathcal E_G^T(\mu,\nu)\Big\}.
\end{equation}

Now let $\eta$ be a curve of maximal slope for $v_0$ with respect to $G$ such that $\eta_0=\mu$.
Since $\eta\in AC^2_{\mathrm{loc}}([0,\infty);\mathcal P_2(E))$,
its endpoint $\eta_T$ belongs to $\mathcal T_T(\mu)$.
Moreover, by the defining identity of a curve of maximal slope,
$$
-\frac{\d}{\d t}v_0(\eta_t)=|\eta'|^2(t)=G(\eta_t)^2
\qquad \text{for a.e. } t\in(0,T).
$$
Integrating from $0$ to $T$, we obtain
$$
v_0(\mu)-v_0(\eta_T)=\int_0^T G(\eta_t)^2\,\d t
= \int_0^T \Big(\frac12 |\eta'|^2(t)+\frac12 G(\eta_t)^2\Big)\,\d t .
$$
Therefore,
$$
\cE_G^T(\mu,\eta_T)
\le
\int_0^T \Big(\frac12 |\eta'|^2(t)+\frac12 G(\eta_t)^2\Big)\,dt
=
v_0(\mu)-v_0(\eta_T).
$$
Hence,
\begin{equation}\label{eq:inf_upperbd}
\inf_{\nu}\Big\{v_0(\nu)+\mathcal E_G^T(\mu,\nu)\Big\}
\le v_0(\eta_T)+\mathcal E_G^T(\mu,\eta_T)
\le v_0(\mu).
\end{equation}
Combining \eqref{eq:inf_lowerbd} and \eqref{eq:inf_upperbd} gives
$$
\inf_{\nu}\Big\{v_0(\nu)+\mathcal E_G^T(\mu,\nu)\Big\}=v_0(\mu).
$$
Moreover, \eqref{eq:inf_upperbd} shows that the infimum is achieved at $\nu=\eta_T$,
which is exactly \eqref{eq:perT_argmin}.
\end{proof}

\section{Early Stopping and Implicit Regularization }
\label{sec:4}
 It is widely observed and analyzed that overparametrized neural networks can overcome overfitting and generalize well \citep{ZhangBengioHardtRechtVinyals2017rethinking,BelkinHsuMaMandal2019doubledescent,neyshabur2017exploring}. One explanation is through implicit regularization of gradient-descent-based training algorithms, see~\cite{Neyshabur2014InSO,heiss2019implicit1}. It is shown that in the training process of large neural networks, by stopping early (compared to training for infinite time),  the trained parameters belong to a set of \enquote{regularized} local minima. In this section, we apply the previous McKean--Vlasov control formulation for gradient flows in overparametrized neural networks, and provide a theoretical analysis of this interesting phenomenon. 
 \subsection{Mean field formulation of neural network}
From now on, let $E:=\Theta:=\R\times\R^d\times\R$ denote the space of parameters, with coordinates $\theta:=(w,a, b)$.  In a one-hidden layer fully-connected neural network with $N$ hidden neurons, with nonlinear activation function $\sigma: \R\rightarrow \R$, the output of the neural network can be represented by
\begin{align*}
    f_N(x):=\frac{1}{N}\sum_{i=1}^N w_i\sigma(a_i\cdot x+b_i),\qquad x\in \R^d,
\end{align*}
where $\{\theta_i:=(w_i,a_i,b_i)\}_{i=1}^N$ denote the trainable parameters. When the neural network is wide enough, its output can be well-approximated by its mean field formulation.
With the same activation function $\sigma$, the mean field formulation of infinite neural network is defined as
\begin{align*}
     f_\mu(x):=\int_{\Theta} w\,\sigma(a\cdot x+b)\,\mu(\mathrm d w, \mathrm d a,\mathrm d b),\qquad x\in\R^d,
\end{align*}
where $\mu\in \cP_2(\Theta)$ is the trainable parameter.

For a given loss function $\tilde L:\R\rightarrow \R_{\ge0}$ and training data set $\{(x_i,y_i)\}_{i=1}^M$,  the empirical loss functional for the mean field neural network is
\begin{align*}
    L(\mu):= \frac{1}{M}\sum_{k=1}^M \tilde L(f_\mu(x_k)-\,y_k).
\end{align*}
Hence, $L$ is a function on $\cP_2(\Theta)$. Without loss of generality, we assume that $\min_{\mu\in \cP_2(\Theta)}L(\mu)=0$. A common choice of the loss function is the quadratic loss function.

In the setting of finite neural networks, the parameters are trained via gradient descent-based algorithms. In the mean field formulation, the analogous formulation is via gradient flow of the loss functional $L$.
\begin{assumption}\label{ass:basic_NN}
\begin{itemize}
    \item[(i)] $\sigma,\sigma'$ are bounded and $\sigma$ is Lipschitz, i.e. there exists a constant $S_0\ge 0$ such that
$$
|\sigma(z)-\sigma(z')|\le S_0|z-z'|,\quad |\sigma(z)|\le S_0,\quad|\sigma'(z)|\le S_0, \quad \forall \, z,z'\in \R.
$$
\item[(ii)] $R:=\max_{1\le k\le M}\|x_k\|<\infty$.
\item[(iii)] $\tilde L:\R\to\R_{\ge 0}\,\in \cC^1(\R)$ and is convex. There exists a constant $\beta>0$ such that
$$
|\tilde L'(r)-\tilde L'(s)|\le \beta|r-s|.\quad \forall\,r,s\in \R.
$$
\item[(iv)] $\min_{r\in \R} \tilde L(r)=0$.
\end{itemize}
\end{assumption}

\begin{remark}
    With the above assumptions, for every $r\in \R$, $(\tilde L'(r))^2 \le 2\beta \ \tilde L(r)$.
\end{remark}
\begin{remark}
    The boundedness assumption on the activation function $\sigma$ can be relaxed to linear growth condition, as long as we assume that $\mu$ has finite fourth moment, i.e. $\mu\in \cP_4(\Theta)$.
\end{remark}
\begin{lemma}
\label{lem: G_property}
    Under Assumption~\ref{ass:basic_NN}, there exists a constant $C_G>0$, depending on all the coefficients in Assumption~\ref{ass:basic_NN}, such that
    $$
    G^2(\mu)\le C_G\ L(\mu)\ (1+m_2(\mu)),\qquad \forall\,\mu\in \cP_2(\Theta).
    $$
    Moreover, the functional $G^2:\mathcal P_2(\Theta)\to[0,\infty]$ is lower semicontinuous with respect to $\cW_2$, i.e. for any $\mu_n\to\mu$ in $\cW_2$, $$
G(\mu)^2\ \le\ \liminf_{n\to\infty} G(\mu_n)^2.$$
\end{lemma}
\begin{corollary}
\label{cor: cEG_lowersemicon}
    For every $\mu\in \cP_2(\Theta)$, the map $\nu\mapsto \cE_G^T(\mu,\nu)$ is lower semicontinuous in $\cW_2$.
\end{corollary}
The proofs are postponed to  Appendix~\ref{apx:proofs}.

\begin{remark}[Bounded approximation of ReLU]
\label{rem:bounded-approx-relu}
For simplicity we assume bounded activations, and extensions to ReLU require moment conditions.
To retain the quadratic-growth framework on $\cP_2(\Theta)$,
we can work with a bounded approximation of ReLU.
A convenient choice is the \emph{clipped ReLU}
$$
\sigma_M(z):=\min\{M,\max\{0,z\}\},\qquad M>0,
$$
which is bounded by $M$ and $1$--Lipschitz.
All results stated under bounded activation apply to $\sigma_M$ for each fixed $M$.
Moreover, $\sigma_M(z)\uparrow \max\{0,z\}$ pointwise as $M\to\infty$.
\end{remark}

\begin{definition}
    For every given initialization $\mu\in \cP_2(\Theta)$, every training dynamic of $L$ starting from $\mu$ is a  curve of maximal slope of $L$ in~Definition~\ref{def: maximal slope}. We also call it the gradient flow of $L$.
\end{definition}

From Theorem~\ref{thm:controlformulation}, for every $\eps\ge0$, the loss functional $L$ admits a McKean--Vlasov control formulation~\eqref{eq:value function}, with running cost identified in~\eqref{eq:ell}.

\subsection{Implicit regularization under infinite training time}
In this section, we aim to provide a mathematical formulation and solution of the open question: 
\begin{quote}
    \textit{Will implicit regularization vanish under infinite training time horizon? In other words, if the gradient descent algorithm converges, does it converge to the global minimizer with minimal \enquote{complexity}?}
\end{quote}

\begin{theorem}[Nonvanishing selection]
\label{thm: nonvanishing_reg}
    If a gradient flow of $L$ converges to $\mu_{\infty}$ in $\cW_2$ as $T\to \infty$, with initial distribution $\mu$, and $\mu_{\infty}\in \arg\min_{\nu\in \cP_2(\Theta)}L(\nu)$, then 
    \begin{align}   \label{eq:selection_claim}\mu_\infty\in\arg\min\Big\{d_G(\mu,\nu):\ \nu\in\arg\min L\Big\}.
    \end{align}
\end{theorem}
\begin{remark}
    Theorem~\ref{thm: nonvanishing_reg} indicates that, even without early stopping, under the mean field formulation of neural networks, if the gradient flow converges, it converges to a global minimizer of the loss function $L$ which minimizes the \emph{weighted} distance with respect to the initial distribution. The \emph{weight} only depends on the choice of the loss function $L$.
\end{remark}

\begin{remark}
    From the result in \cite{chizatbach}, we have that, if $$
    \Phi(\theta,x):=w\sigma(a\cdot x+b)$$
    is 2-homogeneous (i.e. $\Phi(\lambda\theta)=\lambda^2\Phi(\theta)$ for any $\lambda>0$), the initial distribution $\mu$ has full support on $\Theta$, and the gradient flow converges to a limit measure, then the limit must be a global minimizer.
\end{remark}
To prove Theorem~\ref{thm: nonvanishing_reg}, we first state the following lemma, which can be regarded as the limiting behavior of $\cE_G^T$ in~\eqref{eq: finsler_energy} as $T\to\infty$.

\begin{lemma}
\label{lem:energy_conv}
    For every $\nu\in\cP_2(\Theta)$ such that $G(\nu)=0$, we have 
    \begin{equation}
    \label{eq:cE_converge}
         \mathcal E_G^T(\mu,\nu)\downarrow d_G(\mu,\nu)\quad\text{as }\, T\to\infty, \qquad \forall\,\mu\in \cP_2(\Theta).
    \end{equation}
\end{lemma}
\begin{proof}
    For every curve $\eta\in \cA_{T,0}(\mu,\nu)$, we have
    $$
    \int_0^T G(\eta_t)\,|\eta'|(t)\,\d t
\le
\int_0^T\Big(\tfrac12|\eta'|^2(t)+\tfrac12G(\eta_t)^2\Big)\,\d t.
    $$
    Taking the infimum over $\cA_{T,0}(\mu,\nu)$ yields
    $$
    \inf_{\eta\in \cA_{T,0}(\mu,\nu)}\int_0^T G(\eta_t)\,|\eta'|(t)\,\d t\ \le\ \mathcal E_G^T(\mu,\nu).
    $$
    Using Lemma~\ref{lem:reparametrization}, we reparametrize $\eta\in \cA_{T,0}(\mu,\nu)$ to $\tilde\eta\in \cA_{1,0}(\mu,\nu)$ via $\tilde\eta_s:=\eta_{Ts}$ for $s\in [0,1]$.
    Then
    $$
    \int_0^T G(\eta_t)|\eta'|(t)\,\d t
=
\int_0^1 G(\tilde\eta_s)\,|\tilde\eta'|(s)\,\d s.
    $$
    Taking the infimum on the left hand side, we have
    \begin{equation}
    \label{eq:lem:lower_bd}
        d_G(\mu,\nu)\le \cE^T_G(\mu,\nu),\qquad \forall\,T>0.
    \end{equation}
    Fix $\varepsilon>0$. From Proposition~\ref{prop:dg-basic}(iii), there exist $S>0$ and a curve $\gamma\in \cA_{S,0}(\mu,\nu)$ such that
    $$
\int_0^S\Big(\tfrac12|\gamma'|^2(t)+\tfrac12G(\gamma_t)^2\Big)\,\d t
\le d_G(\mu,\nu)+\varepsilon.
$$
For any $T\ge S$, define the concatenation
$$
\eta^T_t=
\begin{cases}
\gamma_t, & 0\le t\le S,\\
\nu, & S<t\le T.
\end{cases}
$$
Since $\eta^T$ is constant on $(S,T]$, we have $|(\eta^T)'|(t)=0$ there. Moreover $G(\nu)=0$
implies $G(\eta^T_t)=0$ on $(S,T]$. Then, for all $T\ge S$,
$$
\cE_G^T(\mu,\nu)\le\int_0^T\Big(\tfrac12|(\eta^T)'|^2(t)+\tfrac12G(\eta^T_t)^2\Big)\,\d t
=
\int_0^S\Big(\tfrac12|\gamma'|^2(t)+\tfrac12G(\gamma_t)^2\Big)\,\d t
\le d_G(\mu,\nu)+\varepsilon.
$$
Hence
$\limsup_{T\to\infty}\cE_G^T(\mu,\nu)\le d_G(\mu,\nu)$. This construction also indicates that, for every $0<T_1< T_2$, we have
$\cE_G^{T_2}(\mu,\nu)\ \le\ \cE_G^{T_1}(\mu,\nu)$.
Combined with the fact that $d_G(\mu,\nu)\le \mathcal E_G^T(\mu,\nu)$ for all $T>0$, we conclude.
\end{proof}

\begin{proof}[Proof of Theorem~\ref{thm: nonvanishing_reg}]
    We first show that $G(\mu_{\infty})=0$ and $\mu_{\infty}\in \arg\min L$. In Proposition~\ref{prop:endpoint_MS}, taking $v_0=L$, we have, for every $T>0$,
    $$
    L(\mu)-L(\mu_T)=\int_0^T G(\mu_t)^2\,\d t,
    $$
    where we used $\|\alpha_t\|_{L^2(\mu_t)}=|\partial L|(\mu_t)=G(\mu_t)$, and $G$ is an upper gradient of $L$. Since $L$ is nonnegative and continuous, taking $T\to \infty$ yields   \begin{equation}\label{eq:G2_integrable_thm}
\int_0^\infty G(\mu_t)^2\,\d t=L(\mu)-\lim_{T\to\infty}L(\mu_T)<\infty.
\end{equation}
Hence there exists a sequence $T_n\to\infty$ such that $G(\mu_{T_n})\to0$. Since $\cW_2(\mu_{T_n},\mu_{\infty
})\to 0$ as $n\to\infty$, together with the lower semicontinuity of $G$ from Lemma~\ref{lem: G_property}, we obtain
$$
G(\mu_\infty)\le\liminf_{n\to\infty}G(\mu_{T_n})=0,
$$
which implies that $G(\mu_\infty)=0$.

Now fix $T>0$. From Proposition~\ref{prop:endpoint_MS}, along every curve of maximal slope $\boldsymbol{\mu}\in \cA_{T,0}(\mu)$,
\begin{equation}\label{eq:perT_argmin_EGT}
\mu_T\in\arg\min_{\nu}\Big\{L(\nu)+\mathcal E_G^T(\mu,\nu)\Big\}.
\end{equation}
Fix any $\bar\mu\in\arg\min L$. From \eqref{eq:perT_argmin_EGT} we have, for every $T>0$,
\begin{equation}\label{eq:compare_bar_mu}
L(\mu_T)+\mathcal E_G^T(\mu,\mu_T)
\le
L(\bar\mu)+\mathcal E_G^T(\mu,\bar\mu).
\end{equation}
Since $L(\mu_T)\ge L(\bar\mu)$, dropping the $L$ term leads to 
\begin{equation}\label{eq:compare_E_only}
\mathcal E_G^T(\mu,\mu_T)\le \mathcal E_G^T(\mu,\bar\mu),\qquad\forall \,T>0.
\end{equation}
Sending $T$ to infinity, from Lemma~\ref{lem:energy_conv} we have
$$
\mathcal E_G^T(\mu,\mu_\infty)\downarrow d_G(\mu,\mu_\infty),
\qquad
\mathcal E_G^T(\mu,\bar\mu)\downarrow d_G(\mu,\bar\mu).
$$
Moreover, for each fixed $T>0$, from Corollary~\ref{cor: cEG_lowersemicon}, the map $\nu\mapsto\mathcal E_G^T(\mu,\nu)$ is lower semicontinuous in $\cW_2$,
hence using $\mu_T\to\mu_\infty$ we have
$$
\mathcal E_G^T(\mu_,\mu_\infty)\le \liminf_{n\to\infty}\mathcal E_G^T(\mu,\mu_{T_n})
\qquad\text{for any }T_n\to\infty.
$$
Choosing $T_n=n$ and combining with \eqref{eq:compare_E_only} yields, for each fixed $T$,
$$
\mathcal E_G^T(\mu,\mu_\infty)
\le
\liminf_{n\to\infty}\mathcal E_G^T(\mu,\mu_n)
\le
\liminf_{n\to\infty}\mathcal E_G^T(\mu,\bar\mu)
=
d_G(\mu,\bar\mu),
$$
where the last equality uses  Lemma~\ref{lem:energy_conv} and $G(\bar\mu)=0$. 

Finally, letting $T\to \infty$ on the left hand side, we get
$$
d_G(\mu,\mu_\infty)\le d_G(\mu,\bar\mu).
$$
Since $\bar\mu\in\arg\min L$ is arbitrary, this proves~\eqref{eq:selection_claim}.
\end{proof}

\subsection{Implicit regularization under short-time horizon}
In this section, we show how to apply our framework to recover the classical results on the equivalence between early stopping and ridge penalty. We need a further assumption on $L$.

\begin{assumption}[Growth condition]
\label{ass:growth condition}
    There exist constants $A_0$ and $A_1$ such that, for all $\mu\in \cP_2(\Theta)$,
    $$
    L(\mu)\le A_0+A_1\,m_2(\mu),
    $$
    where $m_2(\mu):=\int_\Theta |\theta|^2\,\mu(\d\theta)$.
\end{assumption}

\begin{remark}
\label{rmk:L_growth_quad}
If $\tilde L$ is quadratic loss, then 
$L$ satisfies Assumption~\ref{ass:growth condition}. Indeed, one may take
$$
A_0:=\frac{2}{M}\sum_{k=1}^M |y_k|^2,
\qquad
A_1:=2S_0^2,
$$
where $S_0$ is the constant in Assumption~\ref{ass:basic_NN}. This follows from the fact that
for every $k\in\{1,\dots,M\}$, by the boundedness of $\sigma$,
$$
|f_\mu(x_k)|
=\Big|\int_\Theta w\,\sigma(\langle a,x_k\rangle+b)\,\mu(\d\theta)\Big|
\le \int_\Theta |w|\,|\sigma(\langle a,x_k\rangle+b)|\,\mu(\d\theta)
\le S_0\int_\Theta |w|\,\mu(\d\theta).
$$
From Cauchy--Schwarz,
$$
\int_\Theta |w|\,\mu(\d\theta)\le \Big(\int_\Theta |w|^2\,\mu(\d\theta)\Big)^{1/2}\le \Big(\int_\Theta |\theta|^2\,\mu(\d\theta)\Big)^{1/2}=\sqrt{m_2(\mu)}.
$$
Combining the previous displays yields
$$
|f_\mu(x_k)|^2 \le S_0^2\,m_2(\mu).
$$
Finally, 
$$
|f_\mu(x_k)-y_k|^2\le 2|f_\mu(x_k)|^2+2|y_k|^2
\le 2S_0^2\,m_2(\mu)+2|y_k|^2.
$$
Averaging over $k=1,\dots,M$ gives 
$$
L(\mu)=\frac1M\sum_{k=1}^M |f_\mu(x_k)-y_k|^2
\le 2S_0^2\,m_2(\mu)+\frac{2}{M}\sum_{k=1}^M |y_k|^2,
$$
which is the desired bound with the stated constants.
\end{remark}

We start with an observation on $\cE^T_G$.
\begin{lemma}
\label{lem:est_cE_W2}
For every $\mu,\nu\in \cP_2(\Theta)$,
    $$
    \cE_G^T(\mu,\nu)=\inf_{\eta\in \cA_{1,0}(\mu,\nu)} \int_0^1 \frac{1}{2T}|\eta'|^2(t)+\frac{T}{2}G^2(\eta_t)\,\d t.
    $$
    Hence $$
    \cE_G^T(\mu,\nu)\ge \frac{1}{2T}\cW_2^2(\mu,\nu).
    $$
\end{lemma}
The proof follows directly via a change of variables in time. Now denote 
\begin{equation}
\label{eq: regularizedloss}
    \Phi_T(\nu):=L(\nu)+\frac{1}{2T}\cW_2(\mu,\nu)^2,
\end{equation}
where $\mu$ is the initial distribution. We will show that, when $T\ll 1$, along any gradient flow $\boldsymbol{\mu}$ of $L$ starting at $\mu$, $\mu_T$ is an $\eps(T)$-minimizer of $\Phi_T$.
\begin{lemma}
\label{lem: nu_T_est}
    Let $\nu_T\in \arg\min_{\nu\in \cP_2(\Theta)} \Phi_T(\nu)$, then 
    $$
m_2(\nu_T)\ \le\ 2m_2(\mu)\ +\ 4T\,L(\mu).
    $$
\end{lemma}
The proof can be found in Appendix~\ref{apx:proofs}. 
\begin{theorem}
\label{thm: earlystoppingbound}
   There exists a constant $C>0$ such that, along every gradient flow of $L$ with initial distribution $\mu\in \cP_2(\Theta)$, for every $T>0$, 
   \begin{equation}
   \label{eq:early_stopping_est}
        \inf_\nu \Phi_T(\nu)\le\Phi_T(\mu_T)\ \le\ \inf_\nu \Phi_T(\nu)+CT\ r(T,\mu),
   \end{equation}
    where $r(T,\mu):=\big[1+m_2(\mu)+TL(\mu)\big]^2$.
\end{theorem}
\begin{proof}
    Since $G^2\ge 0$, for every $\eta\in \cA_{1,0}(\mu,\nu)$, 
    $$\int_0^1 \frac{1}{2T}|\eta'|^2(t)+\frac{T}{2}G^2(\eta_t)\,\d t\ge \int_0^1 \frac{1}{2T}|\eta'|^2(t)\d t\ge \frac{1}{2T} \cW_2^2(\mu,\nu).
    $$
    Taking the infimum over $\cA_{1,0}(\mu,\nu)$ on the left-hand-side and adding $L(\nu)$ on both sides yields
    $$
    L(\nu)+\cE_{G}^T(\mu,\nu)\ge L(\nu)+\frac{1}{2T}\cW_2^2(\mu,\nu)\ge \Phi_T(\nu),\qquad \forall\, \nu\in \cP_2(\Theta).
    $$
    Hence, we obtain
    $$
    \inf_{\nu} \Phi_T(\nu)\le \inf_{\nu}\{L(\nu)+\cE_G^T(\mu,\nu)\}.
    $$
    For every $\nu$, let $\boldsymbol{\mu}\in \cA_{1,0}(\mu,\nu)$ denote the constant-speed $\cW_2$-geodesic between $\mu$ and $\nu$. Then 
    $$
    \int_0^1\frac{1}{2T} |\mu'|^2(t)\d t=\frac{1}{2T}\cW_2^2(\mu,\nu),\qquad \cE_G^T(\mu,\nu)\le \frac{1}{2T}\cW_2^2(\mu,\nu)+\frac{T}{2}\int_0^1 G^2(\mu_t)\d t.
    $$
    Take $\nu=\nu_T\in \arg\min \Phi_T$ and let $\boldsymbol{\rho}\in \cA_{1,0}(\mu,\nu_T)$ be the corresponding constant speed geodesic. Then we have
    $$
    \inf_{\nu} \Phi_T(\nu)\le \inf_{\nu}\{L(\nu)+\cE_G^T(\mu,\nu)\}\le \inf_{\nu} \Phi_T(\nu)+\frac{T}{2}\int_0^1 G^2(\rho_t)\d t.
    $$
    From Lemma~\ref{lem: G_property} and Assumption~\ref{ass:growth condition}, 
    \begin{align}
    \label{eq: G_est_geo}
        \int_0^1 G(\rho_t)^2\,\d t &\le C_G\int_0^1 L(\rho_t)\big(1+m_2(\rho_t)\big)\,\d t
        \le C_G\int_0^1 (A_0+A_1m_2(\rho_t))(1+m_2(\rho_t))\d t.
    \end{align}
    Notice that the second moment $t\mapsto m_2(\rho_t)$  is convex along geodesics. Then, for every $t\in [0,1]$,
    $$
    m_2(\rho_t)\le (1-t)m_2(\mu)+t\,m_2(\nu_T)\le m_2(\mu)+m_2(\nu_T).
    $$
    Plugging this into~\eqref{eq: G_est_geo} and letting $C:=C_G\max\{A_0,A_1\}$, we obtain
    $$
\frac{T}{2}\int_0^1 G(\rho_t)^2\,dt
\le \frac{CT}{2}\,\left(1+m_2(\mu)+m_2(\nu_T)\right)^2.
    $$
    By Lemma~\ref{lem: nu_T_est}, we then have 
    $$
    \frac{T}{2}\int_0^1 G(\rho_t)^2\,dt
\le CT\big[1+m_2(\mu)+TL(\mu)\big]^2.
    $$
\end{proof}

\subsection{Stability interpretation of the kinetic energy}
\label{subsec:stability}

In the mean--field control formulation, the dynamic evolves according to
$$
\partial_t\mu_t+\nabla_\theta\!\cdot(\mu_t \alpha_t)=0,
  \qquad 
  \alpha_t=-\nabla_\theta\partial_\mu L(\mu_t),
$$
and the  kinetic energy
$$
\mathcal K(T)
  :=\frac12\int_0^T\|\alpha_t\|_{L^2(\mu_t)}^2\,\d t
$$
appears naturally in both the Benamou--Brenier formulation of the Wasserstein
metric and in the Energy--Dissipation Equality (EDE)
\(
L(\mu_0)-L(\mu_T)=\int_0^T\|\alpha_t\|_{L^2(\mu_t)}^2\,\d t.
\)
Beyond its geometric role, $\mathcal K(T)$ admits a direct \emph{stability interpretation} that connects optimization dynamics with generalization.

\paragraph{Algorithmic stability.}
For a training dataset $S=\{z_i\}_{i=1}^M$ and its variant
$S^{(i)}$ obtained by replacing $z_i$ with an independent copy,
let $\boldsymbol \mu^S$ and $\boldsymbol \mu^{S^{(i)}}$ denote the parameter--law trajectories
produced by the algorithm on $S$ and $S^{(i)}$, respectively.
Define the uniform stability at time $T$ by
$$
\mathrm{stab}(T)
  :=\sup_{i\le M,\,z}\;
    \big|
      \ell\big(f_{\mu_T^S}(x),y\big)
      -\ell\big(f_{\mu_T^{S^{(i)}}}(x),y\big)
    \big|.
$$
We also define the expected generalization gap at time $T$ by
$$
\mathrm{gen}(T)
  :=\E_S\Big[\cR(\mu_T^S)-\cR_S(\mu_T^S)\Big].
$$
By the classical stability theorem of \cite{BousquetElisseeff2002,HardtRechtSinger2016},
one has
$$
 |\mathrm{gen}(T)| \le \mathrm{stab}(T).
$$
\begin{remark}
    At a heuristic level, and under smoothness and Lipschitz assumptions,
one expects an estimate of the form
$$
 \mathrm{stab}(T)\le  C_{\rm gen}\cW_2\big(\mu_t^{S},\mu_t^{S^{(i)}}\big)
  \le
  \frac{ C_{\rm gen}C_{\mathrm{stab}}}{M}\int_0^t\|\alpha_s^{S}\|_{L^2(\mu_s^{S})}\,\d s,
$$
where $C_{\rm gen}$ and $C_{\mathrm{stab}}$ are two constants.
Therefore the stability, and hence the generalization gap, are controlled by the
\emph{path length} of the trajectory. The Cauchy--Schwarz inequality yields
$$
 \int_0^T \|\alpha_t\|_{L^2(\mu_t)}\,\d t
  \le
  \sqrt{T}\,
  \Big(\int_0^T \|\alpha_t\|_{L^2(\mu_t)}^2\,\d t\Big)^{1/2}
  =
  \sqrt{2T\,\mathcal K(T)}.
$$
Consequently, under suitable regularity assumptions,
$$
 |\mathrm{gen}(T)|
  \le
  \frac{C_{\mathrm{gen}}C_{\mathrm{stab}}}{M}\sqrt{2T\,\mathcal K(T)}.
$$
Hence the kinetic energy $\int_0^T\| \alpha_t\|^2_{L^2(\mu_t)}\,\d t$
quantifies the cumulative sensitivity of the training trajectory to data perturbations:
shorter and slower trajectories (that is, trajectories with smaller action)
correspond to more stable algorithms and smaller generalization gaps.
\end{remark}

\paragraph{Interpretation.}
In this view, the kinetic term plays two simultaneous roles:
\begin{itemize}
  \item[\textbullet] \emph{Optimization:} it measures the metric speed of the mean--field flow and
    appears in the energy--dissipation identity.
  \item[\textbullet] \emph{Generalization:} through stability bounds, it quantifies how much
    the solution depends on individual data points.
\end{itemize}
Minimizing $\mathcal K(T)$ therefore not only enforces smoother dynamics
(implicit regularization), but also enhances the algorithm's robustness to data perturbations.
This provides a rigorous link between the geometry of gradient flows in Wasserstein space and
the statistical notion of stability, and corresponds to the statement in \cite{HardtRechtSinger2016}:"\emph{If one can achieve low training error quickly on a nonconvex problem with stochastic gradient, our results guarantee that the resulting model generalizes well.}"

\subsection{Bounds on moments}
\begin{proposition}[Second–moment growth]\label{prop:momentcontrol}
Recall that $m_2(\mu):=\int_{\Theta}\|\theta\|^2\,\mu(\d\theta)$. 
Let $\boldsymbol{\mu}$ be the closed-loop flow generated by the feedback control 
$\alpha(\theta,\mu)$ with initial distribution $\mu\in \cP_2(E)$ from~\eqref{eq:continuity_eq}. Then, for almost every $T>0$,
\begin{equation}\label{eq:m2_gronwall}
m_2(\mu_T)
\le e^{T}\,m_2(\mu)+\int_0^T e^{T-t}\|\alpha_t\|^2_{L^2(\mu_t)}\,\d t.
\end{equation}
If $\boldsymbol{\mu}$ is a  gradient flow of $L$, then 
\begin{equation*}   
m_2(\mu_T)
\le e^{T}\,m_2(\mu)+\int_0^T e^{T-t}G(\mu_t)^2\,\d t
\le e^{T}\Big(m_2(\mu)+L(\mu)-L(\mu_T)\Big).
\end{equation*}
\end{proposition}

\begin{proof}
Taking the test function $\varphi(\theta)=|\theta|^2$ in the weak formulation of the continuity equation yields
$$
\frac{\d}{\d t}m_2(\mu_t)=2\int_\Theta \theta\cdot \alpha_t(\theta)\,\mu_t(\d\theta)\le m_2(\mu_t)+\int_\Theta |\alpha_t|^2\,\mu_t(\d\theta),
\qquad\text{for a.e.\ }t.
$$
By Gr\"onwall inequality, we obtain
$$
m_2(\mu_T)
\le e^{T}\,m_2(\mu)+\int_0^T e^{T-t}\Big(\int_\Theta |\alpha_t|^2\,\mu_t(\d\theta)\Big)\,\d t.
$$
Along the curves of maximal slope, with minimum control, we have for a.e. $t$,
$$
\int_\Theta |\alpha_t|^2\,\d\mu_t \;=\; |\mu'|^2(t)\;=\;G(\mu_t)^2.
$$
Hence,
$$
m_2(\mu_T)
\le e^{T}\,m_2(\mu)+\int_0^T e^{T-t}G(\mu_t)^2\,\d t
\le e^{T}\Big(m_2(\mu)+L(\mu)-L(\mu_T)\Big).
$$
\end{proof}

\subsection{Interpretations of implicit regularizations}
\label{sec:Interpretations}
The solution reached by gradient flow minimizes the original loss
subject to a dynamical or action cost, measuring the total movement
of the particle distribution in parameter space.

Moment bounds show that this action cost controls
the parameter complexity of the network.
In one sentence: \emph{implicit regularization} is the preference induced by
the training dynamics, not an explicit penalty that selects among many
interpolating solutions those with small algorithm--dependent complexity, such as 
minimum--norm in linear/NTK limits, maximum--margin in separable classification,
and small--movement/low--action  solutions in the
mean--field feature--learning regime.

\section{Numerical Results}
\label{sec:numerics}
In this section, we test several deterministic diagnostics suggested by the theory in Section~\ref{sec:4}. First, we examine whether the finite-width training dynamics exhibit the energy--dissipation and complexity-growth patterns predicted by the mean-field formulation. Second, we test a finite-width proxy of Theorem~\ref{thm: earlystoppingbound}, which states that for small $T>0$, with initial distribution $\mu $, the gradient-flow endpoint $\mu_T$ is an approximate minimizer of
$$
\Phi_T(\nu)=L(\nu)+\frac{1}{2T}\cW_2^2(\mu,\nu).
$$

\paragraph{Finite-width model and data.}
We work with the finite-width approximation of large neural network, where
$$
f_N(x)=\frac1N\sum_{i=1}^N w_i\,\sigma(a_i\cdot x+b_i),
\qquad \theta_i=(w_i,a_i,b_i)\in\Theta=\R\times\R^{d+1},
$$
so that the empirical parameter distribution is
$$
\mu_t^N=\frac1N\sum_{i=1}^N \delta_{\theta_i^t}.
$$
Since the analysis in Section~4 assumes bounded activations, in the numerics we use the clipped ReLU
$$
\sigma_M(z):=\min\{M,\max\{0,z\}\},
$$
with $M=5$, consistent with Remark~\ref{rem:bounded-approx-relu}. The data are generated from a network of the same form with additive Gaussian noise. Training is performed by deterministic full-batch gradient descent.

\paragraph{Recorded observables.}
Along the trajectory \(t\mapsto \mu_t^N\), we record the empirical loss \(L(\mu_t^N)\), the second moment
$$
m_2(\mu_t^N):=\frac1N\sum_{i=1}^N |\theta_i^t|^2,
$$
and the Barron-type surrogates
$$
B_1(\mu_t^N):=\frac1N\sum_{i=1}^N |w_i^t|,
\qquad
B_{1,1}(\mu_t^N):=\frac1N\sum_{i=1}^N |w_i^t|\,\|a_i^t\|.
$$
To compare with the energy--dissipation identity, we use the mean-field-scaled discrete kinetic action
$$
\sum_k \frac{1}{\Delta t}\sum_{i=1}^N |\theta_i^{k+1}-\theta_i^k|^2,
$$
which is the finite-width approximation of
$
\int_0^T \|\alpha_t\|_{L^2(\mu_t)}^2\,\d t.$ We also use the identity-coupling displacement proxy
$$
\widetilde \cW_{2,\mathrm{id}}^2(\mu_0^N,\mu_t^N)
:=\frac1N\sum_{i=1}^N |\theta_i^t-\theta_i^0|^2,
$$
and the associated short-time proxy
$$
\widetilde\Phi_t^{N,\mathrm{id}}
:=
L(\mu_t^N)+\frac{1}{2t}\widetilde \cW_{2,\mathrm{id}}^2(\mu_0^N,\mu_t^N),
\qquad t>0.
$$

\paragraph{Deterministic scaling diagnostics.}
Figure~\ref{fig:sec5_diag1} displays three deterministic diagnostics along a single full-batch run.
The left panel compares the loss drop \(L(\mu_0^N)-L(\mu_t^N)\) with the cumulative mean-field-scaled kinetic action. The observed parity is close to the diagonal, in agreement with the energy--dissipation interpretation of the gradient-flow dynamics. The middle panel plots the second-moment increment \(m_2(\mu_t^N)-m_2(\mu_0^N)\) against the scale \(\sqrt{t}\sqrt{L(\mu_0^N)-L(\mu_t^N)}\), and the right panel shows the analogous increments for \(B_1\) and \(B_{1,1}\). In this run, both moment growth and Barron-surrogate growth remain controlled and increase smoothly with the dissipated energy. We interpret this as numerical evidence that, in the deterministic mean-field regime, the loss decrease is accompanied by only moderate growth of parameter complexity, consistent with the implicit-regularization mechanism suggested by Proposition~\ref{prop:momentcontrol} and the discussion in Section~\ref{sec:Interpretations}.

\paragraph{Finite-width proxy for Theorem~\ref{thm: earlystoppingbound}.}
We also perform a direct finite-width proxy test of the short-time endpoint principle in Theorem~\ref{thm: earlystoppingbound}. For each small stopping time $T$, starting from the same initialization $\mu_0^N$, we compare two endpoints:
\begin{enumerate}
\item the endpoint \(\mu_T^{N,\mathrm{GF}}\) produced by mean-field-scaled full-batch gradient descent on \(L_N\);
\item an approximate minimizer of the empirical identity-coupling proxy
$$
\widetilde\Phi_T^{N,\mathrm{id}}(\nu)
:=
L_N(\nu)+\frac{1}{2T}\widetilde \cW_{2,\mathrm{id}}^2(\mu_0^N,\nu).
$$
\end{enumerate}
Figure~\ref{fig:sec5_diag2} compares the corresponding displacement proxies from initialization. For small $T$, the displacement of the gradient-flow endpoint and that of the proxy minimizer are of the same order and remain close. This is the finite-width analogue of the statement that the gradient-flow endpoint is approximately optimal for the regularized short-time variational problem. Since the experiment replaces the exact Wasserstein term by the identity-coupling proxy and computes the minimizer only approximately, we view it as supportive rather than definitive evidence for Theorem~\ref{thm: earlystoppingbound}. Nevertheless, it is consistent with the qualitative prediction that, for short stopping times, the endpoint selected by gradient descent is close to the optimizer of a loss-plus-movement functional.

\paragraph{Interpretation.}
Taken together, the numerical experiments support the following picture. In the deterministic mean-field regime, the gradient-flow trajectory satisfies an energy--dissipation balance and exhibits controlled growth of moments and Barron-type surrogates. In addition, for short stopping times, the endpoint generated by gradient descent is well approximated by the minimizer of a finite-width proxy of the regularized loss functional in Theorem~\ref{thm: earlystoppingbound}. These experiments do not constitute a direct computation of $\cE_G^T$, nor a full numerical verification of the continuum theorem, but they do support the variational interpretation of early stopping developed in Sections~\ref{sec:3} and~\ref{sec:4}.

\begin{figure}[t]
  \centering
  \includegraphics[width=\textwidth]{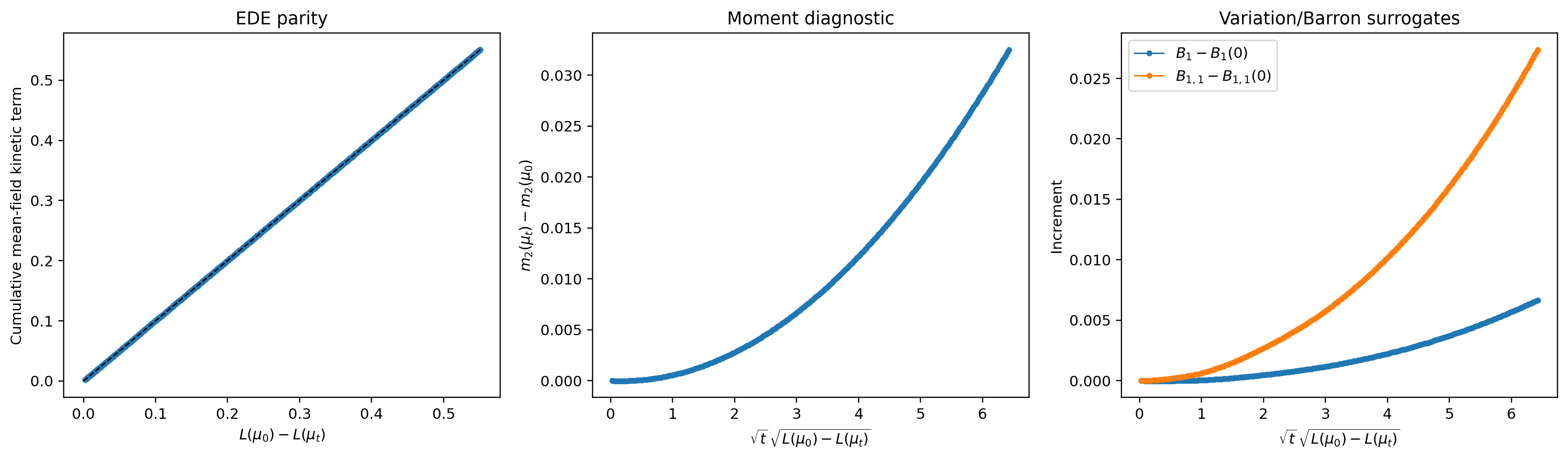}
  \vspace{-1ex}
  \caption{
Deterministic diagnostics for the mean-field finite-width training dynamics.
\textbf{Left:} energy--dissipation parity, comparing the loss drop \(L(\mu_0^N)-L(\mu_t^N)\) with the cumulative mean-field-scaled kinetic action.
\textbf{Middle:} second-moment increment \(m_2(\mu_t^N)-m_2(\mu_0^N)\) plotted against \(\sqrt{t}\sqrt{L(\mu_0^N)-L(\mu_t^N)}\).
\textbf{Right:} analogous increments for the Barron-type surrogates \(B_1\) and \(B_{1,1}\).
The three panels indicate that loss dissipation is accompanied by controlled growth of moments and variation-type quantities.}
\label{fig:sec5_diag1}
\end{figure}

\begin{figure}[t]
    \centering
    \includegraphics[width=0.6\linewidth]{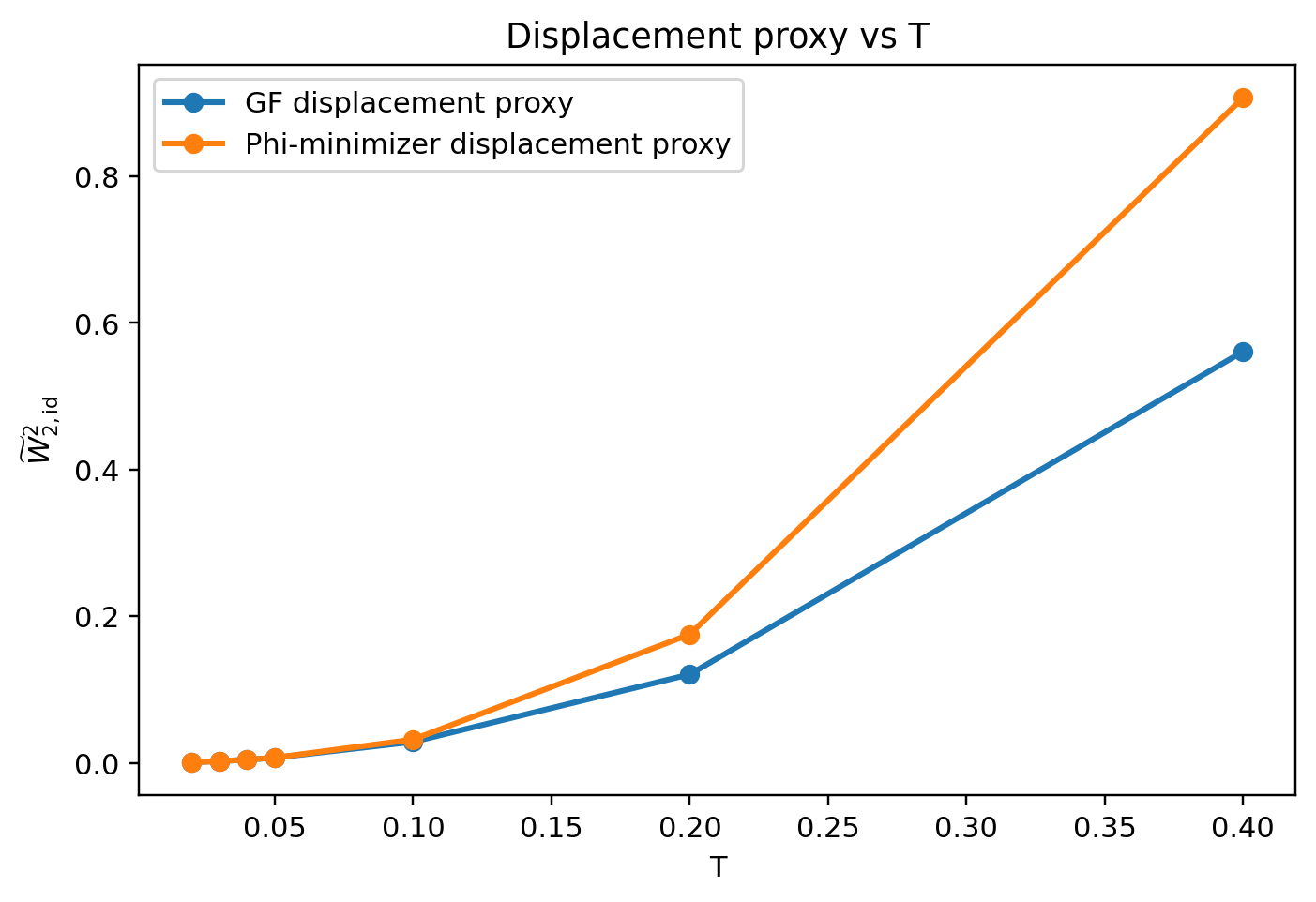}
    \caption{
Finite-width identity-coupling proxy for Theorem~\ref{thm: earlystoppingbound}.
For each small stopping time $T$, we compare the displacement proxy from initialization of the gradient-flow endpoint $\mu_T^{N,\mathrm{GF}}$ and of an approximate minimizer of the proxy functional
$
\widetilde\Phi_T^{N,\mathrm{id}}(\nu)
=
L_N(\nu)+\frac{1}{2T}\widetilde \cW_{2,\mathrm{id}}^2(\mu_0^N,\nu).
$
The two curves remain of the same order for small $T$, consistent with the short-time endpoint principle.}
\label{fig:sec5_diag2}
\end{figure}

\section{Conclusion, Limitations and Future Work}

We proposed a mathematical framework connecting gradient flows and McKean--Vlasov optimal control to study implicit regularization in overparametrized two-layer neural networks. By reformulating the mean-field training dynamics as a control problem, we obtained a variational principle that characterizes early stopping as a trajectory-selection mechanism: the stopped iterate approximately minimizes the loss penalized by a dynamical cost measuring the total movement of the parameter distribution. We further showed that, even under infinite training time, the gradient flow selects a minimizer closest to the initialization in a data- and loss-dependent pseudo-metric~$d_G$, if the gradient flow converges.

Our analysis has several limitations. It is restricted to two-layer networks in the mean-field limit, and extending the framework to deeper architectures remains a substantial open challenge. The implicit regularizer that emerges depends on the loss function and training data through the upper gradient~$G$, making it harder to interpret than classical explicit penalties such as weight decay. Moreover, while we provide non-asymptotic bounds relating the regularization to Wasserstein distances and moment growth, the tightness of these bounds in practical settings is unclear and warrants further investigation. More broadly, this work is primarily theoretical and does not claim to offer immediate practical guidance; we view it as providing a geometric and variational perspective that may prove valuable as the theoretical understanding of neural network training matures.

We see several natural directions for future work. It would be interesting to tighten the connection between~$d_G$ and function-space complexity measures such as Barron or variation norms, potentially yielding more interpretable characterizations of the inductive bias. As discussed in Section~\ref{sec:literature}, a particularly promising direction is to leverage such connections to obtain a more fine-grained understanding of implicit regularization in terms of feature learning, building on recent work relating parameter norms to properties of the learned function.
\bibliographystyle{plainnat}
\bibliography{ref}

\appendix
\section{A Counter Example}
\label{ex: gfnotminimizer}
Let $E:=\R^2$, and consider $(\cP_2(\R^2), \cW_2)$. Let $\eta\in C_c^\infty(\mathbb R)$ be a smooth cutoff with $\eta\equiv1$ on $[-1,1]$, $\eta\equiv0$ on $\mathbb R\setminus[-2,2]$, and $0\le\eta\le1$.
Define
$$
f(x,y):= y\,\eta(y)\in C_c^\infty(\mathbb R^2),\qquad 
S(\mu):=\!\int_{\mathbb R^2} f\,\d\mu,\qquad
\kappa(\mu):=\Big(\int_{\mathbb R^2} |\nabla f|^2\,\d\mu\Big)^{1/2}.
$$
Set a \emph{piecewise-linear} (kinked) scalar nonlinearity
$$
h(z):= a z + |z|,\qquad a>1,
$$
and the value functional
$$
v(\mu):=h(S(\mu))= a\,S(\mu) + |S(\mu)|.
$$
Note that $v$ is proper, l.s.c., Lipschitz on $\cW_2$-balls (since $S$ is Lipschitz via Benamou--Brenier and $f\in C_c^\infty$), but \emph{not} $C^1$ at $S=0$.

\medskip
\noindent\textbf{Relaxed slope and viscosity HJB.}
The metric subdifferential in the $S$-direction at $\mu$ is the interval
$$
\partial h(S(\mu))=\begin{cases}
\{a+1\}, & S(\mu)>0,\\
[a-1,a+1], & S(\mu)=0,\\
\{a-1\}, & S(\mu)<0,
\end{cases}
$$
so the \emph{relaxed slope} of $v$ at $\mu$ is
$$
|\partial^- v|(\mu)\;=\;\min_{s\in\partial h(S(\mu))}\, |s|\,\kappa(\mu)
\;=\;
\begin{cases}
(a+1)\,\kappa(\mu),& S(\mu)>0,\\
(a-1)\,\kappa(\mu),& S(\mu)=0,\\
(a-1)\,\kappa(\mu),& S(\mu)<0.
\end{cases}
$$
Define the running cost by the \emph{viscosity HJB identity} with the relaxed slope, 
$$
\ell(\mu):=v(\mu) + \tfrac12\,|\partial^- v|(\mu)^2.
$$
Then $v$ is a \emph{viscosity} solution of
$$
v(\mu)+\tfrac12\,|\partial^- v|(\mu)^2=\ell(\mu)
$$
on $\mathcal P_2$ (test functionals touch $S(\cdot)$ in the direction of minimal subgradient).

\medskip
\noindent\textbf{Initial condition at the kink.}
Pick $\mu_0\in\mathcal P_2$ such that $S(\mu_0)=0$ and $\kappa(\mu_0)>0$ (e.g., any compactly supported law with zero $f$-moment but nontrivial mass in $\{|\nabla f|>0\}$).
At $S=0$ the subgradient set is the whole interval $[a-1,a+1]$ with \emph{minimal} element $s_{\min}=a-1>0$ and other \emph{larger} elements  $s\in(s_{\min},a+1]$.

\medskip
\noindent\textbf{Two distinct gradient flows (curves of maximal slope).}
From~\cite{ambrosioGradientFlowsMetric2008},  with the relaxed slope, one may construct (locally in time) curves of maximal slope $\mu^{(s)}_t$ calibrated by a \emph{constant} subgradient choice $s\in[a-1,a+1]$ at $t=0$, i.e.
$$
-\frac{\d}{\d t}v(\mu^{(s)}_t)
=\frac12\,|\partial^- v|(\mu^{(s)}_t)^2+\frac12\,|\mu^{(s)}{}'|(t)^2,
\qquad 
|\partial^- v|(\mu^{(s)}_t)= s\,\kappa(\mu^{(s)}_t),
$$
and (by equality in Cauchy--Schwarz/Young for the \emph{relaxed} slope)
$$
|\mu^{(s)}{}'|(t)= s\,\kappa(\mu^{(s)}_t).
$$
In particular, for the two choices
$$
s_{\min}=a-1
\qquad\text{and}\qquad
s_{\max}=a+1,
$$
we obtain \emph{two} gradient flows $\mu^{(\min)}$ and $\mu^{(\max)}$ starting from the same $\mu_0$ and satisfying the curves of maximal slope definition  with the \emph{relaxed} slope.

\medskip
\noindent\textbf{Control costs and non-optimality of the larger-$s$ flow.}
Let $\alpha^{(s)}_t$ be any velocity field realizing the metric speed: 
$|\mu^{(s)}{}'|(t)^2=\int|\alpha^{(s)}_t|^2\,\d\mu^{(s)}_t$.
Along $\mu^{(s)}$ the discounted control cost reads
$$
J(\mu^{(s)},\alpha^{(s)})=\int_0^\infty e^{-t}\Big(\int|\alpha^{(s)}_t|^2\,\d\mu^{(s)}_t+\ell(\mu^{(s)}_t)\Big)\,\d t.
$$
Using the viscosity HJB (with $|\partial^- v|$) and the EDE  (again with $|\partial^- v|$),
\[
\int|\alpha^{(s)}_t|^2\,\d\mu^{(s)}_t
=\big|\mu^{(s)}{}'|(t)^2
=\big(|\partial^- v|(\mu^{(s)}_t)\big)^2
=s^2\,\kappa(\mu^{(s)}_t)^2,
\]
\[
\ell(\mu^{(s)}_t)=v(\mu^{(s)}_t)+\tfrac12\,|\partial^- v|(\mu^{(s)}_t)^2
=v(\mu^{(s)}_t)+\tfrac12 s^2\,\kappa(\mu^{(s)}_t)^2,
\]
and hence
$$
\int|\alpha^{(s)}_t|^2\,\d\mu^{(s)}_t+\ell(\mu^{(s)}_t)
=v(\mu^{(s)}_t)+\tfrac32\,s^2\,\kappa(\mu^{(s)}_t)^2.
$$
Therefore, for $s_{\max}>s_{\min}$ we have \emph{strictly larger} instantaneous kinetic and running cost:
$$
\int|\alpha^{(\max)}_t|^2\,\d\mu^{(\max)}_t+\ell(\mu^{(\max)}_t)
-\Big(\int|\alpha^{(\min)}_t|^2\,\d\mu^{(\min)}_t+\ell(\mu^{(\min)}_t)\Big)
=\tfrac32\,(s_{\max}^2-s_{\min}^2)\,\kappa(\mu_t)^2\;>\;0,
$$
at least for small $t$ while $\kappa(\mu_t)$ stays the same.
After multiplying with the discount factor $e^{-t}$ in time, and integrating over $t$ from $0$ to infinity, we conclude that
$$
J(\mu^{(\max)},\alpha^{(\max)})\;>\;J(\mu^{(\min)},\alpha^{(\min)}).
$$

\medskip
\noindent\textbf{Conclusion.}
Both $\mu^{(\min)}$ and $\mu^{(\max)}$ are \emph{gradient flows} (curves of maximal slope) for $v$ with respect to the \emph{relaxed} slope $|\partial^- v|$,
because $v$ is only used in \emph{viscosity} form and the kink allows multiple $s\in\partial h(0)$.
However, only the flow with the \emph{minimal} selector $s_{\min}=a-1$ minimizes the control cost.
The larger-selector flow $s_{\max}=a+1$ is a gradient flow but \emph{not} a minimizer of the control formulation.

This shows that, at viscosity regularity (with relaxed slopes and without a calibrating chain rule), a gradient flow of $v$ need not be optimal for the control problem.

\section{Proof of Theorem~\ref{thm:gamma}}
\label{apd:gamma}
\begin{lemma}
\label{lem:leps-lowerbound}
    There exist constants $C_0,\,C_1$ such that, for all $\varepsilon\ge0$ and $\mu\in\mathcal P_2(E)$,
\begin{equation}\label{eq:leps-lower}
\ell_\varepsilon(\mu)\ \ge\ -\frac{\varepsilon^2}{2}\big(C_0+C_1\,m_2(\mu)\big).
\end{equation}
\end{lemma}

\begin{proof}
    By the definition of $\ell_\eps$~\eqref{eq:ell},
    \begin{align*}
        \ell_\eps(\mu)&:=V(\mu)+\tfrac12\!\int_E\!\|\nabla_x\partial_\mu V(\mu)\|^2\,\d\mu\;-\;
 \tfrac{\varepsilon^2}{2}\!\int_E\!\Delta_x\partial_\mu V(\mu)\,\d\mu\\
 &\ge -\frac{\eps^2}{2}\int_{E} \Delta_{x}\partial_{\mu}V(\mu)\d\mu\ge -\frac{\eps^2}{2} C(1+m_2(\mu)).
    \end{align*}
\end{proof}

\begin{lemma}
\label{lem:l_0_lower_semi_cont}
    The functional $\ell_0$ is $\cW_2$--lower semicontinuous on $\{\mu:\ m_2(\mu)\le R\}$ for every $R>0$.
\end{lemma}
This follows directly from the definition of $\ell_0$ and the regularity assumption on $L$.

\begin{lemma}[Equi-coercivity]
\label{lem:equi-coercive}
Fix $T>0$ and $\mu\in\cP_2(E)$.
For $\varepsilon\ge0$ define
$$
F_\varepsilon(\nu):=d_{T,\varepsilon}^2(\mu,\nu)+e^{-T}V(\nu).
$$
Then for every $C\in\R$ there exist $\varepsilon_*(T)\in(0,1]$ and $\bar C$ such that,
for all $\varepsilon\in(0,\varepsilon_*(T)]$,
$$
F_\varepsilon(\nu)\le C
\quad\Longrightarrow\quad
m_2(\nu)\le \bar C.
$$
Consequently, the family $(F_\varepsilon)_{\varepsilon\downarrow0}$ is equi-coercive on $(\cP_2(E),\cW_2)$.
\end{lemma}
\begin{proof}
Fix $C\in\R$, $\varepsilon\in(0,1]$, and $\nu\in\cP_2(E)$ with $F_\varepsilon(\nu)\le C$.
Pick $(\boldsymbol \mu,\alpha)\in A_{T,\varepsilon}(\mu,\nu)$ such that
$$
\int_0^T e^{-t}\Big(\frac12\int_E|\alpha_t|^2\,\d\mu_t+\ell_\varepsilon(\mu_t)\Big)\,\d t
\le d_{T,\varepsilon}^2(\mu,\nu)+1 \le C+1.
$$
Set $M(t):=m_2(\mu_t)$. By Ito's formula,
\begin{align*}
    \frac{\d}{\d t}M(t)=2\int_E x\cdot \alpha_t(x)\,\mu_t(\d x)+\varepsilon^2 d
    \le 2\|x\|_{L^2(\mu_t)}\|\alpha_t\|_{L^2(\mu_t)}+\varepsilon^2d.
\end{align*}
Hence, by Cauchy-Schwarz inequality,
$$
\frac{\d}{\d t}M(t)\le M(t)+\int_E|\alpha_t|^2\,\d\mu_t+\varepsilon^2 d.
$$
From Grönwall's inequality,
\begin{equation}
    \sup_{t\in[0,T]}M(t)\le e^T M(0) + e^T\int_0^T\int_E|\alpha_t|^2\,\d\mu_t\,\d t + \varepsilon^2 d e^T.
\label{eq:supM-gronwall}
\end{equation}

From the near-optimality inequality and the lower bound \eqref{eq:leps-lower},
\begin{align*}
\frac12\int_0^T e^{-t}\int_E|\alpha_t|^2\,\d\mu_t\,\d t
&\le C+1 - \int_0^T e^{-t}\ell_\varepsilon(\mu_t)\,\d t\\
&\le C+1 + \frac{\varepsilon^2}{2}\Big(C_0\int_0^T e^{-t}\d t + C_1\int_0^T e^{-t}M(t)\,\d t\Big)\\
&\le C+1 + \frac{\varepsilon^2}{2}\Big(C_0 + C_1\sup_{t\in[0,T]}M(t)\Big).
\end{align*}
Using $e^{-t}\ge e^{-T}$ on $[0,T]$, we have
\begin{equation}
    \int_0^T\int_E|\alpha_t|^2\,\d\mu_t\,\d t
\le e^T\int_0^T e^{-t}\int_E|\alpha_t|^2\,\d\mu_t\,\d t
\le 2e^T(C+1) + \varepsilon^2 e^T\Big(C_0 + C_1\sup_{t\in[0,T]}M(t)\Big).
\label{eq:kinetic-bound}
\end{equation}

Inserting \eqref{eq:kinetic-bound} into \eqref{eq:supM-gronwall},
$$
\sup_{t\in[0,T]}M(t)
\le e^TM(0) + 2e^{2T}(C+1) + \varepsilon^2 e^{2T}C_0
+ \varepsilon^2 e^{2T}C_1\sup_{t\in[0,T]}M(t)
+ \varepsilon^2 d e^T.
$$
Now choose $\varepsilon_*(T)\in(0,1]$ such that $\varepsilon_*^2 e^{2T}C_1\le \frac12$.
Then, for all $\varepsilon\in(0,\varepsilon_*(T)]$,
$$
\sup_{t\in[0,T]}M(t)
\le 2\Big(e^T m_2(\mu) + 2e^{2T}(C+1) + e^{2T}C_0 + d e^T\Big)
=: \bar C.
$$
In particular $m_2(\nu)=M(T)\le \bar C$, which is the desired uniform moment bound.
Finally, uniform second-moment bounds imply $\cW_2$-precompactness of sublevel sets, hence equi-coercivity.
\end{proof}
Now we present the proof of Theorem~\ref{thm:gamma}.
\begin{proof}[Proof of Theorem~\ref{thm:gamma}]
    By Lemma~\ref{lem:equi-coercive}, the family $(F_\varepsilon)_{\varepsilon\downarrow0}$ is equi-coercive
in $(\cP_2(E),\cW_2)$. It remains to prove the $\Gamma$--$\liminf$ and $\Gamma$--$\limsup$ inequalities.

\medskip
\noindent\textbf{Step 1 ($\Gamma$--$\liminf$ inequality).}
Let $\nu_\varepsilon\to\nu$ in $\cW_2$. If $\liminf_{\varepsilon\to0}F_\varepsilon(\nu_\varepsilon)=+\infty$
there is nothing to prove, so assume it is finite. Extract a sequence $\varepsilon_n\downarrow0$ such that
$$
\lim_{n\to\infty}F_{\varepsilon_n}(\nu_{\varepsilon_n})
=\liminf_{\varepsilon\to0}F_\varepsilon(\nu_\varepsilon)<\infty.
$$
For each $n$, choose $(\boldsymbol \mu^n,\alpha^n)\in A_{T,\varepsilon_n}(\mu,\nu_{\varepsilon_n})$ such that
\begin{equation}
    \int_0^T e^{-t}\Big(\frac12\int_E|\alpha_t^n|^2\,\d\mu_t^n+\ell_{\varepsilon_n}(\mu_t^n)\Big)\,\d t
\le d_{T,\varepsilon_n}^2(\mu,\nu_{\varepsilon_n})+\frac1n.
\label{eq:nearmin}
\end{equation}
By the uniform bound on $F_{\varepsilon_n}(\nu_{\varepsilon_n})$ and Lemma~\ref{lem:equi-coercive}, the curves $(\boldsymbol \mu^n)_{n\ge 1}$ have
uniformly bounded second moments and are equi-continuous in $\cW_2$. Hence, up to a subsequence,
\begin{equation}
    \boldsymbol\mu^n \to \boldsymbol\mu \quad\text{in } C([0,T];(\cP_2(E),\cW_2)),
\qquad
\mu_0=\mu,\ \mu_T=\nu.
\label{eq:muconv}
\end{equation}
Define the space--time flux measures $\mathbf m^n:=\mu_t^n\alpha_t^n\,\d t$ on $(0,T)\times E$.
The kinetic part in \eqref{eq:nearmin} yields a uniform bound on $\alpha^n$ in $L^2((0,T)\times E;\mu_t^n \d t)$, hence
$\{\mathbf m^n\}_{n\ge 1}$  is bounded in the dual of $C_c((0,T)\times E;\R^m)$.
Thus, up to a subsequence, $\mathbf m^n\rightharpoonup \mathbf m$ weakly as vector measures.

Passing to the limit in the weak formulation of the Fokker--Planck equations
$$
\partial_t\mu_t^n+\nabla\!\cdot(\mu_t^n\alpha_t^n)=\frac{\varepsilon_n^2}{2}\Delta\mu_t^n,
$$
we obtain, since $\varepsilon_n^2\to0$,
\begin{equation}
    \partial_t\mu_t+\nabla\!\cdot \mathbf m_t=0
\quad\text{in }\mathcal D'((0,T)\times E).
\label{eq:limitCE}
\end{equation}
Moreover, $\mathbf m\ll \mu_t\d t$, so there exists $\alpha\in L^2(\mu_t\d t)$ such that
$\mathbf m_t=\mu_t\alpha_t$; hence $(\boldsymbol\mu,\alpha)\in A_{T,0}(\mu,\nu)$.

By convexity and standard lower semicontinuity,
\begin{equation}
    \int_0^T\int_E|\alpha_t|^2\,\d\mu_t\,\d t
\le \liminf_{n\to\infty}\int_0^T\int_E|\alpha_t^n|^2\,\d\mu_t^n\,\d t.
\label{eq:lsc-kin}
\end{equation}

For the $\Gamma$--$\liminf$ we only need
\begin{equation}
   \liminf_{n\to\infty}\int_0^T e^{-t}\ell_{\varepsilon_n}(\mu_t^n)\,\d t
\ \ge\ \int_0^T e^{-t}\ell_{0}(\mu_t)\,\d t,
\label{eq:lsc-ell} 
\end{equation}
which follows from \eqref{eq:muconv}, the fact that $\ell_{\varepsilon}\to\ell_0$ pointwise,  $\ell_0$ is lower-semicontinuous from Lemma~\ref{lem:l_0_lower_semi_cont},
and $(\ell_\varepsilon)_{\varepsilon\ge0}$ are uniformly bounded from below on moment balls by Lemma~\ref{lem:leps-lowerbound},
allowing an application of Fatou's lemma after extracting a further subsequence if needed.

Combining \eqref{eq:nearmin}--\eqref{eq:lsc-ell} gives
\begin{align*}
d_{T,0}^2(\mu,\nu)
&\le \int_0^T e^{-t}\Big(\frac12\int_E|\alpha_t|^2\,\d\mu_t+\ell_0(\mu_t)\Big)\,\d t\\
&\le \liminf_{n\to\infty} \int_0^T e^{-t}\Big(\frac12\int_E|\alpha_t^n|^2\,\d\mu_t^n+\ell_{\varepsilon_n}(\mu_t^n)\Big)\,\d t
\le \liminf_{n\to\infty} d_{T,\varepsilon_n}^2(\mu,\nu_{\varepsilon_n}).
\end{align*}
Finally, since $V$ is $\cW_2$--l.s.c.\ and $\nu_{\varepsilon_n}\to\nu$,
$$
e^{-T}V(\nu)\le \liminf_{n\to\infty} e^{-T}V(\nu_{\varepsilon_n}).
$$
Therefore,
$$
F_0(\nu)\le \liminf_{n\to\infty}F_{\varepsilon_n}(\nu_{\varepsilon_n})
= \liminf_{\varepsilon\to0}F_\varepsilon(\nu_\varepsilon),
$$
which proves the $\Gamma$--$\liminf$ inequality.

\medskip
\noindent\textbf{Step 2 ($\Gamma$--$\limsup$ inequality).}
Fix $\nu\in\cP_2(E)$ and $\eta>0$. Choose $(\boldsymbol\mu,\alpha)\in A_{T,0}(\mu,\nu)$ such that
\begin{equation}
    \int_0^T e^{-t}\Big(\frac12\int_E|\alpha_t|^2\,\d\mu_t+\ell_0(\mu_t)\Big)\,\d t
\le d_{T,0}^2(\mu,\nu)+\eta.\label{eq:eta-opt}
\end{equation}

We construct admissible pairs for $\varepsilon>0$ with terminal marginals $\nu_\varepsilon\to\nu$.

Let $P_s$ denote the heat semigroup on $E$. 
For $\varepsilon>0$, define
$$
\mu_t^\varepsilon:=P_{\varepsilon^2 t}\mu_t,\qquad
\mathbf m_t^\varepsilon:=P_{\varepsilon^2 t}(\mu_t\alpha_t),\qquad
\nu_\varepsilon:=\mu_T^\varepsilon=P_{\varepsilon^2 T}\nu.
$$
Then $\nu_\varepsilon\to\nu$ in $\cW_2$ as $\varepsilon\to0$.
A standard commutation identity yields
$$
\partial_t\mu_t^\varepsilon + \nabla\!\cdot \mathbf m_t^\varepsilon
=\frac{\varepsilon^2}{2}\Delta\mu_t^\varepsilon
\quad\text{in }\mathcal D'((0,T)\times E),
\qquad
\mu_0^\varepsilon=\mu,\ \mu_T^\varepsilon=\nu_\varepsilon,
$$
so $(\boldsymbol\mu^\varepsilon,\alpha^\varepsilon)\in A_{T,\varepsilon}(\mu,\nu_\varepsilon)$, where
$\mathbf m_t^\varepsilon=\mu_t^\varepsilon\alpha_t^\varepsilon$.

By Jensen's inequality,
for a.e.\ $t$,
\begin{equation}
    \int_E |\alpha_t^\varepsilon|^2\,\d\mu_t^\varepsilon
\le \int_E |\alpha_t|^2\,\d\mu_t.\label{eq:jensen}
\end{equation}
Hence, using \eqref{eq:jensen} and the pointwise convergence $\ell_\varepsilon\to\ell_0$ on moment balls,
together with $\mu_t^\varepsilon\to\mu_t$ in $\cW_2$ uniformly in $t\in[0,T]$, we obtain
\begin{align*}
d_{T,\varepsilon}^2(\mu,\nu_\varepsilon)
&\le \int_0^T e^{-t}\Big(\frac12\int_E|\alpha_t^\varepsilon|^2\,\d\mu_t^\varepsilon+\ell_\varepsilon(\mu_t^\varepsilon)\Big)\,\d t
\le \int_0^T e^{-t}\Big(\frac12\int_E|\alpha_t|^2\,\d\mu_t+\ell_\varepsilon(\mu_t^\varepsilon)\Big)\,\d t\\
&\xrightarrow[\varepsilon\to0]{}\int_0^T e^{-t}\Big(\frac12\int_E|\alpha_t|^2\,\d\mu_t+\ell_0(\mu_t)\Big)\,\d t
\le d_{T,0}^2(\mu,\nu)+\eta,
\end{align*}
where the last inequality is \eqref{eq:eta-opt}. Moreover, $V(\nu_\varepsilon)\to V(\nu)$.
Therefore,
$$
\limsup_{\varepsilon\to0}F_\varepsilon(\nu_\varepsilon)
\le d_{T,0}^2(\mu,\nu)+\eta+e^{-T}V(\nu).
$$
Letting $\eta\downarrow0$ proves the $\Gamma$--$\limsup$ inequality.

\medskip
\noindent\textbf{Step 3 (Convergence of minimizers).}
By Lemma~\ref{lem:equi-coercive}, $(F_\varepsilon)_{\varepsilon\ge0}$ is equi-coercive, and by Steps~1--2 we have
that $F_\varepsilon\ \Gamma$--converges to $F_0$ in $(\cP_2(E),\cW_2)$.
The fundamental theorem of $\Gamma$--convergence then yields:
(i) $\min F_\varepsilon \to \min F_0$, and
(ii) any sequence of minimizers $\nu_{\varepsilon,T}\in\arg\min F_\varepsilon$ is precompact in $\cW_2$,
and every $\cW_2$--limit point is a minimizer of $F_0$.
This proves the claim on convergence of minimizers (along subsequences).
\end{proof}

\section{Postponed proofs}
\label{apx:proofs}
    \begin{proof}[Proof of Proposition~\ref{prop:dg-basic}]
    
    For any admissible curve $\boldsymbol\mu\in\cA_{1,0}(\mu,\nu)$, the integrand
$|\mu'|(t)G(\mu_t)$ is nonnegative a.e., hence the integral is nonegative and thus $d_G(\mu,\nu)\ge0$.

If $\nu=\mu$, the constant curve $\mu_t\equiv \mu$ belongs to $\cA_{1,0}(\mu,\mu)$ and satisfies
$|\mu'|(t)=0$ a.e., so the integral equals $0$. Hence $d_G(\mu,\mu)=0$.

Fix $\mu,\nu$ and let $\boldsymbol\mu=(\mu_t)_{t\in[0,1]}\in\cA_{1,0}(\mu,\nu)$.
Define the time-reversed curve $\bar\mu_t:=\mu_{1-t}$.
Then $\bar\mu\in\cA_{1,0}(\nu,\mu)$.
Moreover, the metric derivative is invariant under time reversal:
$|\bar\mu'|(t)=|\mu'|(1-t)$ for a.e.\ $t\in[0,1]$.
Therefore, by the change of variables $s=1-t$,
\begin{align*}
\int_0^1 |\bar\mu'|(t)\,G(\bar\mu_t)\,dt
&=\int_0^1 |\mu'|(1-t)\,G(\mu_{1-t})\,dt
=\int_0^1 |\mu'|(s)\,G(\mu_s)\,ds.
\end{align*}
Taking the infimum over all $\boldsymbol\mu\in\cA_{1,0}(\mu,\nu)$ yields
$d_G(\nu,\mu)\le d_G(\mu,\nu)$.
By symmetry of the roles of $\mu$ and $\nu$, we also get $d_G(\mu,\nu)\le d_G(\nu,\mu)$, hence equality.

For the triangle inequality, fix $\mu^0,\mu^1,\mu^2$ and let $\delta>0$.
Choose curves $\boldsymbol\mu^{01}\in\cA_{1,0}(\mu^0,\mu^1)$ and
$\boldsymbol\mu^{12}\in\cA_{1,0}(\mu^1,\mu^2)$ such that
\begin{align}
\int_0^1 |{\mu^{01}}'|(t)\,G(\mu^{01}_t)\,dt
&\le d_G(\mu^0,\mu^1)+\delta,\label{eq:delta_min_01}\\
\int_0^1 |{\mu^{12}}'|(t)\,G(\mu^{12}_t)\,dt
&\le d_G(\mu^1,\mu^2)+\delta.\label{eq:delta_min_12}
\end{align}
Define the concatenated curve $\boldsymbol\mu^{02}=(\mu^{02}_t)_{t\in[0,1]}$ by 
$$
\mu^{02}_t:=
\begin{cases}
\mu^{01}_{2t}, & t\in[0,\tfrac12],\\
\mu^{12}_{2t-1}, & t\in[\tfrac12,1].
\end{cases}
$$
Then $\boldsymbol\mu^{02}\in\cA_{1,0}(\mu^0,\mu^2)$. By Lemma~\ref{lem:reparametrization}, for a.e.\ $t\in(0,\tfrac12)$, $|{\mu^{02}}'|(t)=2\,|{\mu^{01}}'|$, and similarly for a.e. $t\in(\tfrac12,1)$, $|{\mu^{02}}'|(t)=2\,|{\mu^{12}}'|(2t-1)$.
Hence, using the change of variables $s=2t$ on $[0,\tfrac12]$ and $s=2t-1$ on $[\tfrac12,1]$,
\begin{align*}
\int_0^1 |{\mu^{02}}'|(t)\,G(\mu^{02}_t)\,dt
&=\int_0^{1/2} 2|{\mu^{01}}'|(2t)\,G(\mu^{01}_{2t})\,dt
  +\int_{1/2}^{1} 2|{\mu^{12}}'|(2t-1)\,G(\mu^{12}_{2t-1})\,dt\\
&=\int_0^{1} |{\mu^{01}}'|(s)\,G(\mu^{01}_{s})\,ds
  +\int_{0}^{1} |{\mu^{12}}'|(s)\,G(\mu^{12}_{s})\,ds.
\end{align*}
Therefore,
$$
d_G(\mu^0,\mu^2)
\le \int_0^1 |{\mu^{02}}'|(t)\,G(\mu^{02}_t)\,dt
\le d_G(\mu^0,\mu^1)+d_G(\mu^1,\mu^2)+2\delta.
$$
Since $\delta>0$ is arbitrary, the triangle inequality follows.

\medskip
For the third point, set $\fE_G(\mu,\nu):=\inf_{T>0} \cE_G^T(\mu,\nu)$. For every $\eta\in \cA_{T,0}(\mu,\nu)$,  we have
\begin{equation}
\label{eq: G_est_cauchy_schwarz}
    G(\eta_t)\,|\eta'|(t)\le \tfrac12|\eta'|^2(t)+\tfrac12G(\eta_t)^2,\quad \forall\, t\in [0,T].
\end{equation}
Now reparametrize $\eta$ to the unit interval by $\gamma_s:=\eta_{Ts}$, $s\in[0,1]$.
Then $\gamma\in \cA_{1,0}(\mu,\nu)$, and $|\gamma'|(s)=T|\eta'|(Ts)$ a.e.. Hence, combined with~\eqref{eq: G_est_cauchy_schwarz} gives
\begin{equation*}
    \int_0^1 G(\gamma_s)\,|\gamma'|(s)\,\d s
=
\int_0^T G(\eta_t)\,|\eta'|(t)\,\d t\le
\int_0^T\Big(\tfrac12|\eta'|^2+\tfrac12G(\eta)^2\Big)\,\d t.
\end{equation*}
Taking the infimum over all such $\eta$ and then over all $T>0$ yields
$d_G(\mu,\nu)\le \mathfrak E_G(\mu,\nu)$.

\medskip
For the other direction, fix $\varepsilon>0$ and take $\gamma\in \cA_{1,0}(\mu,\nu)$ such that
\begin{equation}\label{eq:eps_min_length}
\int_0^1 G(\gamma_s)\,|\gamma'|(s)\,\d s\le d_G(\mu,\nu)+\varepsilon.
\end{equation}
Define the absolutely continuous, nondecreasing map
$$
\theta(s):=\int_0^s \frac{|\gamma'|(r)}{G(\gamma_r)}\,\d r,
$$
with the convention $\frac{|\gamma'|}{G}=0$ on $\{|\gamma'|=0\}$ and $\frac{|\gamma'|}{G}=+\infty$ on
$\{|\gamma'|>0,\ G=0\}$. Let $S:=\theta(1)\in[0,\infty]$.
\\
If $S=\infty$, for every $n\in \N$, define
$$\theta_n(s):=\int_0^s \frac{|\gamma'|(r)}{\max\{G(\gamma_r),\,1/n\}}\,\d r,\qquad s\in[0,1].
$$
Then $\theta_n$ is absolutely continuous, nondecreasing, and $S_n:=\theta_n(1)<\infty$.
Let $s_n:[0,S_n]\to[0,1]$ be an a.e.\ inverse of $\theta_n$ and define the reparametrized curve
$\eta^{(n)}_t:=\gamma_{s_n(t)},\, t\in[0,S_n]$, then $\eta^{(n)}\in \cA_{S_n,0}(\mu,\nu)$. Moreover, 
$$
|\eta^{(n)}|'(t)=|\gamma'|(s_n(t))\,s_n'(t)=\max\{G(\eta^{(n)}_t),\,1/n\}\quad\text{a.e.}
$$
Therefore, together with Lemma~\ref{lem:reparametrization},$$
\int_0^{S_n}\Big(\tfrac12|\eta^{(n)}|'(t)^2+\tfrac12G(\eta^{(n)}_t)^2\Big)\,\d t
\le
\int_0^{S_n} |\eta^{(n)}|'(t)\,G(\eta^{(n)}_t)\,\d t + \frac{S_n}{2n^2}=\int_0^1 G(\gamma_s)\,|\gamma'|(s)\,\d s+ \frac{S_n}{2n^2}.$$
Since $S_n\le n\int_0^1|\gamma'|(s)\,\d s$, we have $S_n/n^2\to0$, and thus
$$
\limsup_{n\to\infty}
\int_0^{S_n}\Big(\tfrac12|\eta^{(n)}|'(t)^2+\tfrac12G(\eta^{(n)}_t)^2\Big)\,\d t
\le \int_0^1 G(\gamma_s)\,|\gamma'|(s)\,\d s.
$$
This yields $\mathfrak E_G(\mu,\nu)\le d_G(\mu,\nu)$.

\end{proof}
\medskip
\begin{proof}[Proof of Lemma~\ref{lem: G_property}]
    Set 
    $$
    r_k=y_k-f_\mu(x_k),\\ \quad v_k(\theta)=\begin{pmatrix}
\sigma(z_k)\\ w\sigma'(z_k)x_k\\ w\sigma'(z_k)
\end{pmatrix},\quad z_k=a\cdot x_k+b.
    $$
    \textbf{Step 1: Growth Condition.}
    Applying Cauchy-Schwarz to $G^2$ yields
    $$
    G^2(\mu)\le \Big(\frac1M\sum_{k=1}^M (\tilde L'(r_k(\mu)))^2\Big)\cdot
\int_\Theta \frac1M\sum_{k=1}^M \|v_k(\theta)\|^2\,\mu(\d\theta).
    $$
    Now, by Assumption~\ref{ass:basic_NN},
    $$
    \|v_k(\theta)\|^2
\le S_0^2 + w^2 S_0^2(\|x_k\|^2+1)
\le S_0^2 (R^2+1) (1+ w^2) .
    $$ 
    Integration over $\theta$ against $\mu$ yields
    $$
G^2(\mu)
\le
2\beta\,L(\mu)\,(R^2+1)S_0^2\Big(1+\int_\Theta w^2\,\mu(d\theta)\Big)\le C_GL(\mu)(1+m_2(\mu)),
    $$
    where $C_G:=2\beta(R^2+1)S_0^2$.
    \medskip
    \par
    \noindent\textbf{Step 2: Lower semicontinuity.} Let $\mu_n\to \mu$ in $\cW_2$. Denoting $ g(\mu,\theta):=\nabla_\theta \partial_\mu L(\mu)(\theta)$, then $G^2(\mu)=\int_\Theta \|g(\mu,\theta)\|^2\,\mu(\d\theta)$. It is easy to check that, for every $R>0$,    $$\lim_{n\to\infty}\sup_{|\theta|\le R}\|g(\mu_n,\theta)-g(\mu,\theta)\|= 0.
    $$
    For every $n\in \N$, take an optimal coupling $\pi_n\in\Pi(\mu_n,\mu)$, and let $(\Theta_n,\Theta)\sim\pi_n$. Then $\Theta_n\to\Theta$ in $L^2$, and hence, along a subsequence,
    $$
    g(\mu_n,\Theta_n)\to g(\mu,\Theta)\qquad\text{a.s.}.
    $$
    By Fatou's lemma, we have
    $$
    \mathbb E\|g(\mu,\Theta)\|^2 \le \liminf_{n\to\infty}\mathbb E\|g(\mu_n,\Theta_n)\|^2.
    $$
    Since $\Theta\sim\mu, \Theta_n\sim\mu_n$, we conclude that
    $$
    G(\mu)^2 \le \liminf_{n\to\infty} G(\mu_n)^2.
    $$
\end{proof}

\begin{proof}[Proof of Corollary~\ref{cor: cEG_lowersemicon}]
Let $\nu_n\to\nu$ in $\cW_2$. If $\liminf_n \mathcal E_G^T(\mu,\nu_n)=+\infty$, there is nothing to prove. Otherwise,  pass to a subsequence (not relabeled) such that $$\sup_{n}\mathcal E_G^T(\mu,\nu_n)<\infty.$$
    For each $n$, pick $\eta^n\in \cA_{T,0}(\mu,\nu_n)$ such that
    $$
    \int_0^T\Big(\tfrac12|\eta^{n}{}'|^2(t)+\tfrac12G^2(\eta^n_t)\Big)\,\d t
\le \mathcal E_G^T(\mu,\nu_n)+\frac1n\le \sup_{n}\mathcal E_G^T(\mu,\nu_n)+1.
    $$
    In particular, the curves $(\eta^n)_{n\ge 1}$  are equicontinuous since
    $$
    \cW_2(\eta^n_t,\eta^n_s)\le \int_s^t |\eta^{n}{}'|(r)\,\d r
\le |t-s|^{1/2}\Big(\int_0^T|\eta^{n}{}'|^2(r)\d r\Big)^{1/2},\quad \forall\, n\ge1.
    $$
    Moreover, since $\eta^n_0=\mu$, we have
    $$
    \cW_2(\mu,\eta^n_t)\le \int_0^t |\eta^{n}{}'|(r)\,\d r \le \sqrt{T}\Big(\int_0^T|\eta^{n}{}'|^2(r)\,\d r\Big)^{1/2}\le\sqrt{T}(\sup_{n}\mathcal E_G^T(\mu,\nu_n)+1).
    $$
    Hence $\{\eta^n_t: n\in\mathbb N,\ t\in[0,T]\}$ is tight with uniformly bounded second moments, thus it is relatively compact in $(\cP_2(\Theta),\cW_2)$.  By Arzelà–Ascoli, there exist a subsequence (not relabeled) and a limit curve $\eta\in \cA_{T,0}(\mu,\nu)$, such that $$
    \sup_{t\in[0,T]} \cW_2(\eta^n_t,\eta_t)\to 0.$$
    From Lemma~\ref{lem: G_property}, 
    $$
    G^2(\eta_t) \le \liminf_{n\to\infty} G^2(\eta^n_t)\qquad\text{for all }\,t\in[0,T],
    $$
    and by Fatou's lemma,
    $$
    \int_0^T G^2(\eta_t)\,\d t \le \liminf_{n\to\infty}\int_0^T G^2(\eta^n_t)\d t.
    $$
    Combined with the fact that $$
    \int_0^T |\eta'|^2(t)\,\d t \le \liminf_{n\to\infty}\int_0^T |\eta^{n}{}'|^2(t)\,\d t,
    $$
    this gives
    $$
    \mathcal E_G^T(\mu,\nu)
\le
\int_0^T\Big(\tfrac12|\eta'|^2(t)+\tfrac12G^2(\eta_t)\Big)\,\d t
\le
\liminf_{n\to\infty}\ \mathcal E_G^T(\mu,\nu_n).
    $$
\end{proof}

\begin{proof}[Proof of Lemma~\ref{lem: nu_T_est}]
    Let $\nu_T$ be the minimizer of $\Phi_T$. Taking the constant curve $\boldsymbol{\eta}\equiv\mu$, we obtain
    $$
    \Phi_T(\nu_T)=L(\nu_T)+\frac{1}{2T}\cW_2(\mu,\nu_T)^2\ \le\ L(\mu)+\frac{1}{2T}\cW_2(\mu,\mu)^2\ \le\ L(\mu),
    $$
    hence
    $$
    \cW_2(\mu,\nu_T)^2 \le 2T\,L(\mu),
    $$
and
$$
m_2(\nu_T)
\ \le\
2m_2(\mu)+2\cW_2(\mu,\nu_T)^2
\ \le\
2m_2(\mu)\ +\ 4T\,L(\mu).
$$
\end{proof}
\end{document}